\documentclass{article}
\usepackage[utf8]{inputenc}
\usepackage[full]{textcomp}
\usepackage[osf]{newtxtext}
\usepackage{comment}

\usepackage{amssymb}
\usepackage{mathtools}
\usepackage{hyperref}
\usepackage{breakurl}
\usepackage{mhenvs}
\usepackage{mhequ}
\usepackage{mhsymb}
\usepackage{booktabs}
\usepackage{tikz}
\usepackage{mathrsfs}
\usepackage{longtable}
\usepackage{times}

\newtheorem*{lemma**}{Lemma}
\newtheorem*{theorem**}{Theorem}

\usepackage{microtype}
\usepackage{wasysym}
\usepackage{centernot}
\usepackage{enumitem}

\numberwithin{equation}{section}

\makeatletter
\newcommand{\globalcolor}[1]{%
  \color{#1}\global\let\default@color\current@color
}
\makeatother

\newif\ifdark
\IfFileExists{dark}{\darktrue}{\darkfalse}

\ifdark

\definecolor{darkred}{rgb}{0.9,0.2,0.2}
\definecolor{darkblue}{rgb}{0.7,0.3,1}
\definecolor{darkgreen}{rgb}{0.1,0.9,0.1}
\definecolor{pagebackground}{rgb}{.15,.21,.18}
\definecolor{pageforeground}{rgb}{.84,.84,.85}
\pagecolor{pagebackground}
\AtBeginDocument{\globalcolor{pageforeground}}

\else

\definecolor{darkred}{rgb}{0.7,0.1,0.1}
\definecolor{darkblue}{rgb}{0.4,0.1,0.8}
\definecolor{darkgreen}{rgb}{0.1,0.7,0.1}
\definecolor{pagebackground}{rgb}{1,1,1}
\definecolor{pageforeground}{rgb}{0,0,0}

\fi

\theoremnumbering{Alph}
\renewtheorem{theorem*}{Theorem}

\DeclareMathAlphabet{\mathbbm}{U}{bbm}{m}{n}
\DeclareFontFamily{U}{BOONDOX-calo}{\skewchar\font=45 }
\DeclareFontShape{U}{BOONDOX-calo}{m}{n}{
  <-> s*[1.05] BOONDOX-r-calo}{}
\DeclareFontShape{U}{BOONDOX-calo}{b}{n}{
  <-> s*[1.05] BOONDOX-b-calo}{}
\DeclareMathAlphabet{\mcb}{U}{BOONDOX-calo}{m}{n}
\SetMathAlphabet{\mcb}{bold}{U}{BOONDOX-calo}{b}{n}

\let\epsilon\varepsilon
\def\E{{\symb E}}
\def\F{{\mathcal F}}
\def\H{{\mathscr H}}
\def\FC{\mathscr{C}}

\def\X{\mathbb{X}}

\def\X{{\mathbf X}}
\def\XX{{\mathbb X}}

\def\C{\mathcal{C}}
\def\f{\frac}
\def\1{\mathbf{1}}

\def\cov{{\mathrm{Cov}}}
\def\${|\!|\!|}

\def\<{\langle}
\def\>{\rangle}
\setlist{noitemsep,topsep=4pt}
\def\para_#1{/\!\!/_{\!#1}}
\def\var{\mathrm{var}}

\def\slash{\kern0.18em/\penalty\exhyphenpenalty\kern0.18em}
\def\dash{\kern0.18em--\penalty\exhyphenpenalty\kern0.18em}

\makeatletter 
\newcommand*{\fat}{}
\DeclareRobustCommand*{\fat}{%
\mathbin{\mathpalette\bigcdot@{}}}
\newcommand*{\bigcdot@scalefactor}{.5}
\newcommand*{\bigcdot@widthfactor}{1.15}
\newcommand*{\bigcdot@}[2]{%
  \sbox0{$#1\vcenter{}$}
  \sbox2{$#1\cdot\m@th$}%
  \hbox to \bigcdot@widthfactor\wd2{%
    \hfil
    \raise\ht0\hbox{%
      \scalebox{\bigcdot@scalefactor}{%
        \lower\ht0\hbox{$#1\bullet\m@th$}%
      }%
    }%
    \hfil
  }%
}
\makeatother
{\theorembodyfont{\rmfamily}

\newtheorem{convention}[lemma]{Convention}
}

\newtheorem{assumption}[lemma]{Assumption}

\makeatletter
\newcommand\blfootnote[1]{%
  \begingroup
  \renewcommand\thefootnote{}\footnote{#1}%
  \addtocounter{footnote}{-1}%
  \endgroup
}
\makeatother

\begin{document}

\author{Johann Gehringer and Xue-Mei Li \\Imperial College London \blfootnote { johann.gehringer18@imperial.ac.uk,  xue-mei.li@imperial.ac.uk}} 
	
\title{Functional limit theorems for the fractional Ornstein-Uhlenbeck process}

\maketitle

\begin{abstract}
We prove a functional limit  theorem for vector-valued functionals of the fractional Ornstein-Uhlenbeck process, providing the foundation for the fluctuation theory of slow/fast systems driven by both long and short range dependent noise.  The limit process has  both Gaussian and non-Gaussian components. 
The theorem holds for any $L^2$ functions, whereas for functions with stronger integrability properties the convergence is shown to hold in the H\"older topology,
the rough topology for processes in $C^{\f 12+}$. This leads to  a  `rough creation' / `rough homogenization' theorem, by which  we mean the weak convergence of a family of random smooth curves to a non-Markovian random process with non-differentiable  sample paths. In particular, we obtain effective dynamics for  the second order problem and for the kinetic fractional Brownian motion model.
\end{abstract}

{ \scriptsize {\it  keywords:} passive tracer, fractional noise, multi-scale, 
 mixed functional central and non-central limit theorems, rough creation, rough homogenization,   rough topology  }


\setcounter{tocdepth}{2}
\tableofcontents
\section{Introduction}
The functional limit theorem we study here lays the foundation for the fluctuation problem for a slow/fast system
with the fast variable given by  a family of  non-strong mixing stochastic processes, this will be discussed in \cite{Gehringer-Li-tagged}, see 
\cite{Gehringer-Li-homo} for the preliminary version.
 A pivot  theorem for obtaining effective dynamics  for the slow moving particles in a fast turbulent environment are scaling limit theorems for the convergence of the following functionals
\begin{equation}
X^\epsilon:=\left( X^{1,\epsilon}, \dots, X^{N,\epsilon}\right), \qquad 
X^{k,\epsilon}= \alpha_k (\epsilon)\int_0^{t} G_k(y^{\epsilon}_s)ds,
\end{equation}
with weak convergence in $\C^\alpha([0,T],\R)$, where $T$ is some finite fixed time horizon, $\alpha(\epsilon)$ a suitable scaling and $G_k : \R \to \R$.
If $y_t^\epsilon=y_{\f t \epsilon}$ and $y_t$ is a strong mixing process, $\alpha(\epsilon)=\f 1{ \sqrt \epsilon}$ and the limit is
 a Markov process, for details  see e.g. the book \cite{Komorowski-Landim-Olla} and the references therein.
For stochastic  processes whose auto-correlation function does not decay sufficiently fast at  infinity  there is no reason to have  the  $\sqrt \epsilon$ scaling or to obtain a diffusive limit. Furthermore,  the scaling limit and the
limit function may depend on the individual functions~$G_k$.

In this article we take $y_t^\epsilon$ to be the stationary and rescaled fractional Ornstein-Uhlenbeck process with Hurst parameter $H$, which, for $H> \f 1 2$, exhibits long range dependence (LRD) and is not strong-mixing. Our interest for  long range dependent   / non-strong mixing noise comes the time series data of the river Nile. In a study of water flows of the Nile river, Hurst and his colleagues \cite{Hurst} observed long range time dependence and found that the time dependence varies proportionally to $t^H$ where $H\sim 0.73$, by contrast, Brownian motions and stable processes have independent increments.  
Fractional Brownian motions (fBM) were then proposed by Benoit Mandelbrot and John Van Ness \cite{Mandelbrot-VanNess} for modelling the Hurst phenomenon. 
A fBM is a continuous mean zero Gaussian process with stationary increments and covariance $\E(B_t-B_s)^2=|t-s|^{2H}$, it is self-similar with similarity exponent $H$, and distinguished by the Gaussian property and stationary, but dependent increments. See e.g. \cite{Mishura} and the reference therein for stochastic calculus for fBM's.

Self-similar processes   appeared also in mathematically rigorous descriptions of critical phenomena and in renormalisation theory. In \cite{Sinai},  Sinai  constructed non-Gaussian self-similar fields; 
while Dobrushin \cite{Dobrushin} studied self-similar fields subordinated to self-similar Gaussian fields (multiple  Wiener integrals).  
Those self-similar stochastic processes with stationary increments are a particular interesting class. When normalized  to begin at $0$, to have mean $0$ and to have variance $1$ at $t=1$, they necessarily have the covariance  $\f 12(t^{2H}+s^{2H}-|t-s|^{2H})$. Those of Gaussian variety are fBMs.  Hermite processes are non-Gaussian self-similar processes with stationary increments and the above mentioned covariance.
They appeared as scaling limits of functionals of long range dependent Gaussian processes,  see \cite{Rosenblatt}. Jona-Lasinio was also concerned with the construction of a systematic theory of limit distributions for sums of `strongly dependent' random variables for which the classical central limit theorems do not hold, \cite{Jona-Lasinio}, see also the book \cite{Embrechts-Maejima}.

Let us first consider the convergence of one single component, the scalar case.
The scaling constant $\alpha_k(\epsilon)$ depends on the function $G_k$, and is a reflection of the self-similarity exponents of the limiting process. 
If $\alpha_k(\epsilon)=\f 1 {\sqrt \epsilon}$, the limit of $X^{k, \epsilon}$ is a Wiener process and the functional central limit theorem  is expected to hold.  Let  $\mu$ denote the centred and normalized Gaussian measure on $\R$ and $m_k$ denote the Hermite rank of a centred function $G_k \in L^2(\mu)$, 
which is the smallest non zero term in its chaos expansion. Let $(H,m)\mapsto H^*(m)$ denote the function given by (\ref{H-star}),  which decreases with  $m$.  Then,  the relevant scaling constants are given as below:
\begin{equation}\label{beta}
\begin{aligned}
&\alpha\left(\epsilon,H^*(m_k)\right) = \left\{\begin{array}{cl}
\f 1 {\sqrt{\epsilon}}, \, \quad  &\text{ if } \, H^*(m_k)< \f 1 2,\\
\f  1 {\sqrt{  \epsilon \vert \ln\left( \epsilon \right) \vert}}, \, \quad  &\text{ if } \, H^*(m_k)= \f 1 2, \\
\epsilon^{H^*(m)-1}, \quad  \, &\text{ if } \,  H^*(m_k) > \f 1 2.
\end{array}\right.
\end{aligned}
\end{equation}
See Lemma \ref{Integrals} for a preliminary computation indicating the scales.   In the past the limit theorems for the non-Gaussian limits had been called non-central limit theorems, we use the terminology `functional central limit theorems' for all cases.

The intuition for this comes from its counter part for sequences. If $Y_n$ is a mean zero, stationary,  and strong mixing sequence, such that $\sigma_n^2=\E (\sum_{i=1}^n Y_i)^2\to \infty$, $ \E(\sum_{i=1}^n Y_i)^4=O(\sigma_n^4)$, 
then $ \f 1 {\sigma_n} \sum_{i=1}^n Y_i {\longrightarrow}  N(0,1)$. 
If $Y_n$ is not strong mixing, this CLT may fail. 
Indeed, if $X_n$ is a stationary mean zero variance $1$ Gaussian sequence with  auto-correlation $r(n)\sim c {n^{-\gamma}}$ for some $\gamma\in  (0, 1)$ and $c\in \R$ (allowing for negative correlations),  $G$  a centred function with Hermite rank $m\ge 1$, and $A(n)$  a sequence such that
$$\lim_{n\to \infty} \var\left(\f 1 {A(n)} \sum_{k=1}^n G(X_k)\right) =1.$$
Then, $z_n=\f 1 {A(n)} \sum_{k=1}^{[nt]} G(X_n)$ is expected to converge in finite dimensional distributions.
The scaling constant $A(n)$ is of the order $n^{1-\f 12 \gamma m}$ in the long range dependent case, of order $\sqrt n$ in the short range dependent case, and of order $\sqrt {n \ln n}$ for the borderline case, see \cite{Breuer-Major,Bai-Taqqu}. 
By long range dependence, we mean $\sum_{n=1}^\infty |r(n)|= + \infty$.

The limit process, $\lim_{n\to \infty} z_n$,   is a Wiener process for fast decaying correlations, i.e.
in case  $\gamma\in  (\f 1 m, 1)$, \cite{Breuer-Major}.
In the borderline case, $\gamma=\f 1m$,  the scaling limit is also a Wiener process.
However if  $\gamma \in (0, \f 1m)$ the correlations fail to decay sufficiently fast,  the scaling limit is a Hermite process in 
the $m$-th chaos,  \cite{  Dobrushin, Bai-Taqqu}.
The first convergence to a non-Gaussian similar process was shown in \cite{Rosenblatt} where the aim was to construct a not strong mixing sequence of random variables, he achieved this by showing the sequence of random variables has a non-Gaussian scaling limit which is now known as the Rosenblatt process. In \cite{Bai-Taqqu} vector valued combinations of 
short and long range dependent sums were studied, however, the limit of each component is assumed to be moment determinate (which can only happen when they are in the $L^2$ chaos expansion of order less or equal to $2$). This is due to a restriction in the asymptotic independence result in \cite{Nourdin-Rosinski}, which was extended in \cite{Nourdin-Nualart-Peccati}. 

We return to the continuous functional limit theorems.
For the scalar case, the continuous version CLT for $\gamma \in ( \f 1 m ,1 )$ was obtained in \cite{BenHariz}, the borderline case $\gamma = \f 1 m$  in \cite{Buchmann-Ngai-bordercase}. These are shown for the convergence in finite dimensional distribution and for $G$ to be a centred $L^2$ function. They also obtained uniform convergences in the continuous topology  for a restrictive class of functions $G$ (assuming sufficiently fast decay of the coefficients in the Wiener chaos expansion).   
This was extended in \cite{Nourdin-Nualart-Zintout} to vector valued $X^\epsilon$, when each component of $X^\epsilon$ falls in the Brownian case, with convergence understood in the sense of finite dimensional distributions.  The result in \cite{Nourdin-Nualart-Zintout} was improved in  \cite{Campese-Nourdin-Nualart}, where the fast chaos decay restrictions on $G_k$, for $G_k \in L^p$ for $p>2$, are removed with techniques from Malliavin calculus.
In the continuous long range dependent case Taqqu, \cite{Taqqu}, obtained convergence  in the continuous topology. 
These results, although fragmented (in some regimes these are only known  for scalar valued processes or only at the level of sequences), provide a fairly good picture of what is going on. 

There exists however no vector valued functional limit theorem with  joint convergence, when the scaling limit of the components are mixed, in this article we provide a complete description for the joint convergence of $\{X^{k,\epsilon}\}$ for  $G_k \in L^2(\mu)$. We have a functional limit theorem for vector valued processes whose components may exhibit both short and long range dependence.  For $G_k$ satisfying a stronger integrability condition, we can also show weak convergence in the $\C^\alpha([0,T],\R^d)$-topology and for each fixed time in $L^2$ for the low Hermite rank case, which already have interesting applications. 
Furthermore, they are the basis for the convergence in a suitable rough topology, which due to the change of the nature of the problem will appear in  \cite{Gehringer-Li-homo} where rough path theory is used to study slow/fast systems, leading to `rough creation' / `rough homogenization' in which the effective limit is not necessarily a Markov process.

{\it Application.} Consider  the second order equation on $\R$:
$$\begin{aligned}
\dot x_t^\epsilon  &= \epsilon^{H-1} f(x^{\epsilon}_t) y_t^\epsilon, \quad  x_0^\epsilon=x_0\\
 dy_t^\epsilon&=-\f 1 \epsilon y_t^\epsilon dt + \f {\sigma} {\epsilon^H} dB_t^H, \, \, \, \, \, \, \, \, y_0 \sim \epsilon^{-H} \sigma \int_{-\infty}^0 e^{\f { t-s } {\epsilon} } dB^H_s.\end{aligned}$$
Taking $\epsilon \to 0$,  does $x_t^\epsilon$ converge?  In case $H=\f 12$ and $f=1$, this is essentially the Kramer-Smolouchowski limit (this is also called the kinetic Brownian motion model).  For $H\not =\f 12$ and for $f=1$ this was shown in \cite{Boufoussi-Ciprian, Zhang-08,Al-Talibi-Hilbert} to converge to a fBM, see also \cite{Friz-Gassiat-Lyons} for the case with a magnetic field. Given $H>\f 13$ and $f\in C_b^3$ (for any $H$ if $f=1$), we can show $x_t^\epsilon$ converges to the solution of the equation $\dot x_t=f(x_t) \;d B_t^H $ with initial value $x_0$ where the integral in the differential equation is interpreted as a Riemann-Stieltjes integral.
Furthermore we obtain the following bound in $C^{\gamma'}$ where $0<\gamma'<\gamma < H$:
$$\left\|  |x_\cdot^\epsilon -x_\cdot|_{\C^{\gamma'}([0,T])}  \; \right\|_{L^p}  \lesssim T^{\gamma}  \epsilon^{H-\gamma},$$
This computation is straightforward, see Propositions \ref{th-example} and \ref{prop-linear-driver} for detail.

 With the functional limit theorem below, Theorem A, we can conclude also the convergence of solutions of the equations, for $h \in \C^{2}_b(\R^d,\R^d)$ and $g \in \C_b(\R,\R)$,
   $$\dot x_t^\epsilon= \alpha(\epsilon)  f(x_t^\epsilon) \, G(y^\epsilon_t)+ h(x_t^\epsilon) \, g(y^\epsilon_t).$$
  We show that
$ x_t^\epsilon$ converges in $\C^\gamma([0,T],\R)$ for $\gamma \in (0, H^*(m) \vee \f 1 2 - \f 1 p)$
 either to the solution of the equation 
 $$d\bar x_t = c f(\bar x_t) \,dZ_t^{H^*(m),m}+\bar g\; h(\bar x_t),$$ 
where $Z_t^{H^*(m),m}$ is a  Hermite process,  or  to the solution to the Stratonovich stochastic differential equation
$$d\bar x_t = c  f(\bar x_t) \circ \,dW_t+\bar g\; h(\bar x_t) $$ where $W_t$ is a standard Wiener process  (given enough integrability on $G$). Here $c$ is a specific constant (c.f. equation (\ref{c-square}) depending on $G$ arising from the homogenization procedure. 
For the above we follow \cite{Campese-Nourdin-Nualart}  and use Malliavin calculus to obtain suitable moment bounds on $\int_0^t G(y^{\epsilon}_s) ds$. These results appeared in the previous version of the current paper \cite{Gehringer-Li-homo}.  Equations driven by fractional Brownian motions are also studied for the averaging regime, see \cite{Hairer-Li} and  \cite{Fannjiang-Komorowski-2000}.  A fluctuation theorem around the effective average was obtained  \cite{bourguin2019typical}.

{Main Results}
We denote our underlying probability space by $(\Omega, \mathcal{F},\P)$. Let $\mu$ denote the standard Gaussian distribution and we choose $\sigma$ such that the stationary scaled fOU process, to be defined below, satisfies $y^{\epsilon}_t \sim \mu$. Let $\{H_m, m\ge 0\}$  be the orthogonal Hermite polynomials  on $L^2( \mu)$, such that they  have leading coefficient $1$ and $L^2(\mu)$ norm $\sqrt {m!} $.
Given  $G\in L^2(\mu)$, then it posses an expansion of the form
$
G(x)=\sum_{k=0}^\infty c_k H_k(x)$,
where $ c_k=\f 1 {k!}\<G, H_k\>_{L^2(\mu)}$. A function $G$ is centred  if and only if $c_0=0$. The smallest $m$ with $c_m\not =0$ is called the Hermite rank of $G$. 
In case the correlations of $y^{\epsilon}_t$ do not decay sufficiently fast  the path integral $\alpha (\epsilon)\int_0^t G( y_s^\epsilon)ds$ ought to be approximated by that of the first term of its Wiener chaos expansion. By orthogonality of the $H_m$'s it is sufficient to study the asymptotics of $\alpha (\epsilon)\int_0^t H_m( y_s^\epsilon)ds$ to deduce $\alpha(\epsilon)$.

Although the solutions to the fOU equation converge exponentially fast to each other, their autocorrelation function decays  only algebraically. The  indicator for the behaviour of $\alpha(\epsilon) \int_0^t H_m(y_s^\epsilon) ds$ turns out to be
\begin{equation}\label{H-star}
H^*(m) = m(H-1)+1,
\end{equation}
and the self-similarity exponent of the limiting process is determined by $\alpha(\epsilon,H^*(m))$.
For large $m$, the limit will be a Wiener process,  and otherwise the limit $Z_t$ should have the scaling property: 
$\epsilon^{H^*(m)} Z_{\f t\epsilon} \sim Z_t$. Indeed, $Z_t$ are the self-similar Hermite processes.
To state the functional limit theorem concisely, we make the following convention,

\begin{convention}\label{convention}
	Given a collection of functions $(G_k \in L^2(\mu), k\le N)$, we will label the high rank ones first,  so the first $n$ functions
	satisfy  $H^*(m_k) \le \f 12$, where $n \geq 0$,  and the remaining satisfy $H^*(m_k) > \f 1 2$. \end{convention}

 \begin{theorem*}\label{theorem-CLT} 
Let  $y^\epsilon$ be the stationary solution to the scaled fractional Ornstein-Uhlenbeck equation (\ref{fOU}) with standard Gaussian distribution $\mu$. Let $G_k:\R\to \R$ be  centred functions in $L^2(\mu)$ with Hermite ranks $m_k$. Write
$$G_k= \sum_{l=m_k}^\infty  c_{k,l} H_l, \qquad \alpha_k(\epsilon)=\alpha(H^*(m_k), m_k), \quad X^{k,\epsilon}= \alpha_k (\epsilon)\int_0^{t} G_k(y^{\epsilon}_s)ds.$$
Set   $$ X^{W, \epsilon}=\left( X^{1,\epsilon}, \dots, X^{n,\epsilon}\right), \qquad
X^{Z, \epsilon}=\left( X^{n+1,\epsilon}, \dots, X^{N,\epsilon}\right).$$
 Then, the following holds:
\begin{enumerate}
\item \begin{itemize}
\item [(a)] 
There exist stochastic processes $X^W=( X^1, \dots, X^n)$  and $X^Z=(X^{n+1}, \dots, X^N)$ such that on every finite interval $[0,T]$,
$$(X^{W,\epsilon}, X^{Z,\epsilon}) \longrightarrow  (X^W, X^Z),$$
in the sense of finite dimensional distributions.
Furthermore, for any $t >0$
  $$\lim_{\epsilon \to 0} \|X^{Z,\epsilon}_t \to X^Z_t\|_{L^2(\Omega)}=0.$$
 \item [(b)]  If furthermore each $G_k$ satisfies Assumption \ref{assumption-single-scale-not-continuous} below, the convergence is weakly  in $\C^{\gamma}([0,T],\R^N)$  for every $\gamma < \f 1 2 - \f {1} {\min_{k \leq n } p_k}$, if there is at least one component converging to a Wiener process. Otherwise they converge in  $\C^{\gamma}([0,T],\R^N)$ for every $\gamma < \min_{k>n} H^*(m_k) - \f 1 {p_k} $.

    \end{itemize} 
    \item
   We now describe the limit  $X= (X^W,X^Z)$. 
 \begin{enumerate}
 \item [(1)]   $X^W \in \R^n$  and $X^Z \in \R^{N-n}$ are independent. 
\item [(2)]  $ X^W = U \hat W_t$ where $ \hat W_t$ is a standard Wiener process and $U$ is a square root of
the matrix $(A^{i,j})$, 
$$A^{i,j}=\int_0^{\infty} \E\left( G_i(y_s) G_j(y_0) \right) ds = 
  \sum_{q=m_i\vee m_j}^{\infty}   c_{i,q}\; c_{j,q}  \; (q!) \, \int_0^\infty  \rho(r)^q\, dr$$
and $\rho(r)=\E (y_ry_0)$.  In other words,  $\E\left( X^i_t X^j_s\right)= 2 (t \wedge s) A^{i,j}$ for $i,j\le n$.

\item [(3)] Let $Z_t^{H^*(m_k),m_k}$ be the Hermite processes, represented by (\ref{Hermite}).
Then,
$$ X^Z=(c_{n+1,m_{n+1}}  Z_t^{n+1} , \dots, c_{N,m_{N}}  Z_t^{N}),$$   where,
  \begin{equation}\label{Hermite-2}
   Z_t^{k}= \f{m_k!}{K(H^*(m_k),m_k)} Z_t^{H^*(m_k),m_k}.
\end{equation}
We emphasize that the Wiener process $W_t$ defining the Hermite processes is the same for every $k$ (c.f. equation (\ref{Hermite})),
which is in addition independent of $\hat W_t$. 
  \end{enumerate}
\end{enumerate}
 \end{theorem*}

\begin{assumption}[Functional Limit $\C^\gamma$ assumptions]
\label{assumption-single-scale-not-continuous}
Let $G_k\in L^{2}(\mu)$ with Hermite rank $m_k\ge~1$.
\begin{itemize}
 \item  { \it High rank case. }  \quad  If  $H^*(m_k) \le  \f 1 2$, assume $G_k \in L^{p_k}(\mu)$  where $\f 1 2 - \f 1 {p_k} > \f 1 3$ (i.e. $p_k >6$).
\item  {\it  Low rank case.  } \quad  If $H^*(m_k) > \f 1 2$,  assume $G_k \in L^{p_k}(\mu)$  where $ H^*(m_k) - \f 1 {p_k} > \f 1 2$.
 \end{itemize}
\end{assumption}

\begin{remark}
\
The case $H=\f 12$  is classical, and is not of interest here.
In this case the result  is independent of the Hermite rank and the scaling is given by $\alpha(\epsilon)=\f 1{\sqrt \epsilon}$, due to the exponential decay of correlations.
\end{remark}

\bigskip

An immediate application is the following rough homogenisation theorem for a toy model:
\begin{theorem*} 
Let $H \in (\f 1 3,1) \setminus \{ \f 1 2 \}$,  $f\in \C_b^3(\R^d, \R^d)$, $h \in \C^{2}_b(\R^d;\R^d)$, $G \in \C(\R,\R)$ satisfying Assumption \ref{assumption-single-scale-not-continuous} and $g \in \C_b(\R;\R)$.
Let  $\alpha(\epsilon) = \alpha(\epsilon,H^*(m))$. Fix a finite time $T$ and consider
\begin{equation}\label{limit-eq}
\dot x_t^\epsilon =\alpha(\epsilon) f(x_t^\epsilon) G(y^{\epsilon}_t)+ h(x^{\epsilon}_t)g(y^{\epsilon}_t), 
  \qquad  x_0^\epsilon=x_0.
 \end{equation} 
\begin{enumerate}
\item If
$H^*(m) > \f 1 2$, $x_t^\epsilon$ converges weakly in $\C^{\gamma}([0,T],\R^d)$  to the solution to the Young differential equation
$d\bar x_t = c f(\bar x_t) \,dZ_t^{H^*(m),m} +\bar{g} h(\bar{x}_t) dt$  with initial value $x_0$ for $\gamma \in (0, H^*(m)- \f 1 p)$.
\item If
$H^*(m) \leq \f 1 2$, 
 $x_t^\epsilon$ converges weakly  in  $\C^\gamma([0,T],\R)$ to the solution of the Stratonovich stochastic differential equation
$d\bar x_t = c  f(\bar x_t) \circ \,dW_t + \bar{g} h(\bar{x}_t) dt$ with $ \bar x_0=x_0$, where $\gamma \in(0, \f 1 2- \f 1 p)$.
\end{enumerate} 
\end{theorem*}

We also take the liberty to point out an intermediate result on the joint convergence of stochastic processes in finite $L^2$ chaos, for it  maybe of service.  The proof is a slight modification of results in  \cite{Ustunel-Zakai,Nourdin-Nualart-Peccati}. 
\bigskip

{\bf  Proposition \ref{proposition-spit-independence}.}
Let $q_1 \leq q_2 \leq \dots \leq q_n \leq p_1 \leq p_2 \leq \dots \le p_m$. Let  $f_i^\epsilon \in L^2(\R^{p_i})$, $
g_i^\epsilon \in L^2(\R^{q_i})$, $F^{\epsilon}=\left(I_{p_1}(f^{\epsilon}_1), \dots , I_{p_m}(f^{\epsilon}_m)\right)$ and  $G^{\epsilon}=\left(I_{q_1}(g^{\epsilon}_1), \dots, I_{q_n}(g^{\epsilon}_n)\right)$, where $I_q$ denotes the Wiener integral of order $q$.  Suppose that 
 for every $i,j$, and any $1 \leq r \leq q_i$:
$$ \Vert f^{\epsilon}_j \otimes_r g^{\epsilon}_i \Vert \to 0.$$
Then $F^\epsilon \to U$ and $G^{\epsilon} \to V$ weakly imply that $(F^{\epsilon},G^{\epsilon}) \to (U,V)$ jointly, where  $U$ and $V$ are taken to be independent random variables.

\section{Preliminaries}\label{preliminary}
We take the Hermite polynomials of degree $m$ to be $$H_m(x) = (-1)^m e^{\f {x^2} {2}} \f {d^m} {dx^m} e^{\f {x^2} {2}}.$$ Thus, $H_0(x)=1$,  $H_1(x)=x$.   
Let $\hat H$  be the inverse of $H^*(m)=m(H-1)+1$: $$ \hat H(m)=\f  1 m (H-1) + 1.$$
\subsection{Hermite processes}

The rank $1$  Hermite processes $Z^{H, 1}$ are fractional BMs,  the formulation (\ref{Hermite}) below is  exactly the 
 Mandelbrot-Van Ness representation for a fBM. 

\begin{definition}\label{Hermite-processes}
Let $m\in \N$ with $\hat H(m)>\f 12$. The class of {\it Hermite processes} of rank $m$ is given by the following mean-zero processes,
\begin{equation}\label{Hermite}
Z_t^{H,m}=\f  {K(H,m)} {m!} \int_{\R^m} \int_0^t \prod_{j=1}^m (s-\xi_j)_+^{  -(\f 1 2 + \f {1-H} {m})} \, ds \,  d  W({\xi_1}) \dots d  W({\xi_m}),
\end{equation}
where the constant $K(H,m)$ is chosen so their variances are $1$ at $t=1$. \end{definition}
The integration in (\ref{Hermite}) is understood as a multiple Wiener-It\^o integral (over the region $\R^n$ without the diagonal). Note $\hat H(1)=H$.  
 
By the properties of Wiener integrals, two Hermite processes  $Z^{H, m}$ and $Z^{H',m'}$, defined by the same Wiener process, are uncorrelated if $m \not = m'$.
The Hermite processes have stationary increments and finite moments of all orders with covariance 	\begin{equation}
	\E(  Z_t^{H,m} Z_s^{H,m}) =  \f 1 2 (   t^{2H} + s^{2H} - \vert t-s \vert^{2H}).
	\end{equation}
Therefore, using Kolmogorov's theorem \ref{Kolmogorov-theorem}, one can show that the Hermite processes $Z_t^{H,m}$ have sample paths of H\"older regularity up to  $H$. 
They are also self similar with exponent~$H$ which means
$ \lambda^H  Z^{H,m} _{\f \cdot\lambda} \sim  Z^{H,m}_.$.

\begin{remark}
 In some literature,  see e.g. \cite{Maejima-Ciprian} where further details on Hermite processes can also be found,
	the Hermite processes are defined with a different exponent as below:  
	$$\tilde Z_t^{H,m} =\f  {K(H,m)} {m!} \int_{\R^m} \int_0^t \prod_{j=1}^m (s-\xi_j)_+^{ H- \f 3 2} \, ds \,  d  W({\xi_1}) \dots d  W({\xi_m}).$$
The two notions are related by
 \begin{equation}Z_t^{H^*(m), m}=\tilde Z_t^{H,m}, \qquad Z_t^{H,m}=\tilde Z_t^{\hat H(m),m}.
\end{equation}
We refer to \cite{Pipiras-Taqqu-book, Samorodnitsky, Cheridito-Kawaguchi-Maejima} for detailed studies of fBM's which are used in this article.
\end{remark}

\subsection{Fractional Ornstein-Uhlenbeck processes}\label{OU-section}
Let us normalise the fractional Brownian motion, so that $B^H_0=0$ and $\E(B^H_1)^2=1$.  
Disjoint increments of $B_t^H$ have a covariance of the form:
\begin{equation*}
\E(B_t-B_s)(B_u-B_v)=\f 12 \left( |t-v|^{2H}+|s-u|^{2H}-|t-u|^{2H}-|s-v|^{2H} \right).
\end{equation*}
We define the stationary fractional Ornstein-Uhlenbeck processes to be
$y_t =  \sigma \int_{-\infty}^t e^{-\sigma(t-s) } dB^H_s$,
where  $B^H_t$ is a  two-sided fractional BM, $\sigma$ is chosen such that  $y_t$ is distributed as $\mu=N(0,1)$.
It  is the  solution of the following Langevin equation:
$$dy_t = - y_t dt + \sigma d B^H_t, \qquad y_0 = \sigma \int_{-\infty}^0 e^{ s }  dB^H_s.$$ 
We take $y_t^\epsilon$, the fast or rescaled  fOU, to be the  stationary solution of
 \begin{equation}\label{fOU}
dy_t^\epsilon = -\f 1 \epsilon y_t^\epsilon\, dt + \f { \sigma} {{\epsilon}^H}\, d B^H_t.
\end{equation}
Observe that $ y_\cdot ^\epsilon$ and $ y_{\f \cdot \epsilon} $ have the same distributions, 
$y^\epsilon_t=\f \sigma {\epsilon^H}\int_{-\infty}^t e^{-\f 1 \epsilon (t-s) } dB^H_s$. In particular, both $y_t$ and $y^{\epsilon}_t$ are H\"older continuous with H\"older exponents $\gamma \in (0,H)$.
Let  us denote their correlation functions by $\rho$ and $\rho^\epsilon$:
$$\rho(s,t):=\E(y_sy_t), \qquad \rho^{\epsilon}(s,t):= \E(y^{\epsilon}_sy^{\epsilon}_t).$$ 
Let $\rho(s)=\E (y_0y_s)$ for $s\ge 0$  and extended to $\R$ by symmetry, then $\rho(s,t)=\rho(t-s)$ and similarly for $\rho^{\epsilon}$.
For $H>\f 12$, the set of functions for which Wiener integrals are defined include $L^2$ functions and
so $\rho$ posses an analytical expression.
 Indeed, since
$$\E (B^H_tB^H_s)=H(2H-1) \int_0^t\int_0^s |r_1-r_2|^{2H-2}dr_1dr_2,$$
we have
 $$\f{ \partial^2}{\partial t\partial s} \E(B^H_tB^H_s) =H(2H-1)|t-s|^{2H-2},$$ which is integrable, and therefore we may use the Wiener isometry to compute the covariances
\begin{equs}
	\E(y_ty_s) =\sigma^2 H(2H-1) \int_{-\infty}^t\int_{-\infty}^s  e^{-(s+t-r_1-r_2) }  |r_1-r_2|^{2H-2} dr_1 dr_2.
\end{equs}
For $u>0$, we set
$$\rho(u)=\sigma^2 H(2H-1) \int_{-\infty}^{u}\int_{-\infty}^0  e^{-(u-r_1-r_2) }  |r_1-r_2|^{2H-2} dr_1 dr_2.$$ 
Using this, the following correlation decay was shown in \cite{Cheridito-Kawaguchi-Maejima}.
\begin{lemma}
	\label{correlation-lemma}
	Let  $H\in (0, \f 12)\cup (\f 12, 1)$. Then,
	$\rho(s)=2\sigma^2 H(2H-1) s^{2H-2} +O(s^{2H-4})$ as $ s \to \infty$. In particular, for any $t\not =s$,
	 \begin{equation}\label{cor1}
	|\rho(s,t)| \lesssim 1\wedge |t-s|^{2H-2}.
	\end{equation}
\end{lemma}

Hence,  $  \int_0^\infty \rho^m(s) ds$ is finite if and only if $ H^*(m)<\f 12$, or  $H=\f 12$ and $m\in \N$, as in the latter the usual OU process admits exponential decay of correlations.
\begin{lemma}\label{Integrals}
Let $H\in (0,1) \setminus \{ \f 1 2\}$, fix a finite time horizon $T$, then for $t \in [0,T]$ the following holds \emph{uniformly} for $\epsilon \in (0,\f 1 2]$:
\begin{equation} \label{correlation-decay-2-1}
	\left( \int_0^{\f t \epsilon} \int_0^{\f t \epsilon}  \vert  \rho(u,r) \vert^m\, dr \,du\right)^{\f 12} \\
 \lesssim 
\left\{	\begin{array}{lc}
 \sqrt {\f t \epsilon  \int_0^\infty \vert \rho(s) \vert^m } ds,  \quad  &\hbox {if} \quad H^*(m)<\f 12,\\
 \sqrt { (\f t \epsilon)  \vert \ln\left(\f 1 \epsilon \right) \vert}, \quad  &\hbox {if} \quad H^*(m)=\f 12,\\ 
 \left(  \f t \epsilon\right) ^{H^*(m)},  \quad &\hbox {if} \quad H^*(m)>\f 12.
 \end{array} \right.
\end{equation}

\begin{equation} \label{correlation-decay-2-2}
\left( \int_0^{t} \int_0^{t}  \vert  \rho^{\epsilon}(u,r) \vert^m\, dr \,du\right)^{\f 12} \\
\lesssim 
\left\{	\begin{array}{lc}
\sqrt { t \epsilon  \int_0^\infty \vert \rho(s) \vert^m } ds,  \quad  &\hbox {if} \quad H^*(m)<\f 12,\\
\sqrt { t \epsilon  \vert \ln\left(\f 1 \epsilon \right) \vert}, \quad  &\hbox {if} \quad H^*(m)=\f 12,\\
\left(  \f t \epsilon\right) ^{H^*(m)-1},  \quad &\hbox {if} \quad H^*(m)>\f 12.
\end{array} \right.
\end{equation}	
Note, if $H=\f 12$, for and any $m \in \N$, the  bound is always $ \sqrt {\f t \epsilon}  \int_0^\infty \rho^m(s) ds$.	
In particular,
\begin{equation}\label{integral-10}
 t \int_0^{t} \vert \rho^{\epsilon}(s)\vert^m ds 
\lesssim \f { t^{ \left(2H^*(m) \vee 1\right)}}  { \alpha \left( \epsilon, H^*(m)\right)^2}.
\end{equation}
\end{lemma}

\begin{proof}
	We first observe that
	\begin{equation}\label{correlation-decay-3}
	\int_0^\infty \rho^m(s) ds <\infty \quad \Longleftrightarrow \quad  H^*(m)<\f 12\quad 
	\Longleftrightarrow \quad H<1 -\f 1{2m}.
	\end{equation}	
	By a change of variables and using  estimate (\ref{cor1}) on the decay of the auto correlation function,
	\begin{align*}
	\int_0^{\f t \epsilon} \int_0^{\f t \epsilon} \vert \rho(\vert u-r \vert)\vert^m dr du 
	&=2 \f t \epsilon  \int_0^{\f t \epsilon} \vert \rho(s)\vert^m ds\\
	&\lesssim  \left\{	\begin{array}{cl} 
	\f t \epsilon \int_0^\infty \rho^m(s) ds,  &\hbox{ if }H^*(m)<\f 12,\\
	\left( \f t \epsilon \right)^{2H^*(m)}, &\hbox{ otherwise.}
	\end{array} \right. ,
	\end{align*}
	For the case $H^*(m)=\f 1 2$ we use 
	\begin{align*}
	\int_0^{\f t \epsilon} \vert \rho(s) \vert^m ds &\leq \int_0^{\f T \epsilon} \vert \rho(s) \vert^m ds
	\lesssim  \int_0^{\f T \epsilon} \left(1 \wedge \f 1 s\right) ds \lesssim    \big\vert \ln \left( \f T \epsilon \right) \big\vert 
	\lesssim  \big\vert \ln \left( \f 1 \epsilon \right) \big\vert.
	\end{align*}
	To complete the proof we observe that by a simple change of variables,
	\begin{align*}
	\int_0^{t} \int_0^{t}  \vert  \rho^{\epsilon}(u,r) \vert^m\, dr \,du &= \epsilon^2 \int_0^{\f t \epsilon} \int_0^{\f t \epsilon}  \vert  \rho(u,r) \vert^m\, dr \,du.
	\end{align*}
\end{proof}

Next, we recall the Garsia-Rodemich-Romsey-Kolmogorv inequality.
\begin{lemma}\label{lemma-GRR}
	[Garsia-Rodemich-Romsey-Kolmogorov inequality]
	Let $T>0$.
		Let $\theta: [0,T]\to \R^d$. For any positive numbers $\gamma, p$,  there exists a constant $C(\gamma, p)$ such that 
		$$\sup_{s\not =t , s,t \in [0,T]} \f{|\theta(t)-\theta (s)|}{|t-s|^\gamma} \le C(\gamma, p)
		\left(  \int_0^T \int_0^T \f { |\theta_s-\theta_r|^p}{ {|s-r|}^{\gamma p+2}}ds dr\right)^{\f 1p}.$$
\end{lemma}
See \cite[section A.2]{Friz-Victoir} in the Appendix,  as well as  \cite{Stroock-Varadhan-mult-dim-diffusion-processes}, for a proof.
As a consequence of this inequality one obtains the following theorem.
\begin{theorem}[Kolmogorov's Theorem]\label{Kolmogorov-theorem}
Let $\theta$ be a stochastic process.
Suppose that for $s,t\in [0, T]$, $p>1$ and $\delta>0$, 
$$\E|\theta(t)-\theta (s)|^p \le c_p |t-s|^{1+\delta},$$ where $c_p$ is a constant. Then, for $\gamma<\f \delta p$,  $\theta \in \C^{\gamma}([0,T],\R)$,  and in particular
$$\||\theta|_{C^\gamma([0,T])}\|_p \le C(\gamma, p) ( c_p)^{\f 1p} \left(\int_0^T \int_0^T |u-v|^{\delta -\gamma p-1} dudv\right)^{\f 1p},
$$ where right hand side is finite when $\gamma \in (0, \f \delta p)$.
\end{theorem}

\begin{lemma}\label{cty-lemma}
For any $\gamma \in (0, H)$, $p>1$,
the following estimates hold:  
$$\sup_{s, t\ge 0}\f{\|y_s-y_t\|_{L^p}}{ 1 \wedge |s-t|^{H}} \lesssim  1, \qquad  \E  \sup_{s\not =t, s,t\in [0,T]} \left( \f{  |y_s-y_t |} { |t-s|^{\gamma}} \right)^p \lesssim T\;  C(\gamma, p).$$
\end{lemma}
\begin{proof}
	We  use the fractional Ornstein-Uhlenbeck equation
	$$y_s-y_r=-  \int_r^s y_udu+B^H_s-B^H_r,$$
	to obtain $\E|y_s-y_r|^2 \lesssim   (s-r)^2 \E \vert y_1 \vert ^2+q|s-r|^{2H}$.
	Using the stationarity of $y_t$,  one also has  $ \E|y_s-y_r|^{2} \leq  2\E |y_1|^2=2$.
	Since for Gaussian random variables the $L^2$ norm controls the $L^p$ norm we have
	$$\|y_s-y_r\|_{L^p} \lesssim \left\{\begin{aligned} 
	& 1,  &\hbox{ if } |s-r|\ge 1;\\
	& |s-r|^{H}, \quad &\hbox{ if }l |s-r|\le 1. \end{aligned}\right.$$
Thus, by symmetry and a change of variables, 
$\int_0^T \int_0^T \f { \E|y_s-y_r|^p}{ {|s-r|}^{ \gamma p+2}}dsdr\lesssim T$ and application of Kolmogorov's theorem \ref{Kolmogorov-theorem} concludes the proof.
\end{proof}

\section{Applications}
\subsection{The second order problem}
 \label{example-fou}
If $x$ is a stochastic process, we write $x_{s,t}=x_t-x_s$.

\begin{proposition}\label{th-example}
 Let $H\in (0,1)$, $\gamma \in (0,H)$, $p>1$ and fix a finite time~$T$. 
Let $X^{\epsilon}_t = \epsilon^{H-1} \int_0^{t } y^\epsilon _{s} ds$, then, 
$$\sup_{s, t \in [0,T]} \left \| X_{s,t}^\epsilon- \sigma B^H_{s,t}\right\|_{L^p} \lesssim \epsilon^H, \qquad 
   \left\| \left|X^\epsilon - \sigma B^H\right|_{\C^{\gamma'}([0,t],\R)}\right\|_{L^p} \lesssim  t^{\gamma} \epsilon^{H-\gamma},$$
 for any $\gamma'<\gamma<H$ and for $t \in [0,T]$.
\end{proposition}
\begin{proof}
Set  
$v_t^\epsilon=\epsilon ^{H-1}  y^\epsilon_{t}$,  then, $v_t^\epsilon$ solves the following equation
$$dv_t^\epsilon=-\f 1 \epsilon v_t^\epsilon dt + \f {\sigma} {\epsilon} d B^H_t.$$
Using the equation for $v_t^\epsilon$ we have
$$X_{s,t}^\epsilon 
= \epsilon^{H-1} \int_s^{t } y^\epsilon _{r} dr  =\int_s^t v_r^\epsilon dr =\epsilon (v^{\epsilon}_t-v_s^\epsilon) + \sigma B^H_{s,t}.$$
Therefore, for any $p>1$,
\begin{align*}
\sup_{s,t \in [0,T]} \left \| X_{s,t}^\epsilon- \sigma B^H_{s,t}\right\|_{L^p}  &=\sup_{s,t \in [0,T]} \left \| \epsilon (v_t^\epsilon-v_s^{\epsilon})\right\|_{L^p} 
=\epsilon ^{H} \sup_{s,t \in [0,T]} \left\| y^{\epsilon}_{t}- y^{\epsilon}_s \right\|_{L^p} \lesssim \epsilon^H.
\end{align*}
In the last step we used the stationarity  of $ y_t^\epsilon$. 

Thus, applying Kolmogorov's theorem \ref{Kolmogorov-theorem} to $X^\epsilon- \sigma B^H$,  we see that the following holds for any $t \in [0,T]$, $p>1$ and any $\gamma'<\gamma$
$$\||X^\epsilon-\sigma B^H|_{\C^{\gamma'}([0,t],\R)}\|_{L^p} \lesssim \epsilon^H \left({\f t \epsilon}\right)^{\gamma},
$$  
hence, the claim follows. 
\end{proof}
As an application we consider the  following slow/fast system,
\begin{equation}\label{example}
\left\{ \begin{aligned}
&\dot x_t^\epsilon= \epsilon^{H-1} f(x_t^\epsilon) \, y^\epsilon_t,   
\\& x_0^\epsilon=x_0.
\end{aligned}
\right.
\end{equation}
To describe its limit, we review the concept of Young integrals.  If  $f, g: [0, T]\to \R$  with  $f\in \CC^\alpha$ and $g\in \CC^\beta$ such that  $\alpha +\beta>1$,
then the Riemann-Stieljes integral makes sense, and
$$\int_0^t f_s dg_s= \lim_{|\CP|\to 0} \sum_{[u,v] \subset \CP}  f_u(g_u-g_v)\in \CC^\beta.$$
For details see\cite{Young36}; this integral is called a Young integral. 
Since Young integrals have the regularity of its integrand,  for $H>\f 12$  the equation
$\dot x_t=f(x_t) \;dB^H_t$ makes sense.
In \cite{Lyons94}, it was shown that if $f \in \C_b^3$, then the equation has  a unique global solution from each initial value.
This type of equations are Young equation,  the simplest type of rough differential equation.  The notation $\C_b^3$ denotes the space of bounded functions such that their first three derivatives are  bounded.

\begin{proposition}\label{prop-linear-driver}
  Let  $H \in ( \f 1 3, 1)$ and $f\in \C_b^3(\R^d,\R^d)$.  Then for any $\gamma \in (0,H)$, $\gamma' < \gamma$,  $x_t^\epsilon$  converges in $L^p$
 in $\C^{\gamma'}([0,T], \R^d)$
 to the solution of the rough differential equation:
  \begin{equation}\dot x_t= \sigma f(x_t) \;dB^H_t, \, \, \, \, \, \, x_0 =x_0.\end{equation}
Furthermore, for $t \in [0,T]$,
$$\left\|  |x^\epsilon -x|_{\C^{\gamma'}([0,t],\R)}  \; \right\|_{L^p}  \lesssim t^{\gamma}  \epsilon^{H-\gamma}.$$ 
\end{proposition}
\begin{proof}
The idea is to consider  equation (\ref{example}) as a  Young/rough differential equation.
In case $H \in (\f 1 2,1)$, we can rewrite our equation as
$$\dot x_t^\epsilon =f(x_t^\epsilon) dX_t^\epsilon,$$
where $X^{\epsilon}=\epsilon^{H-1} \int_0^{t } y^\epsilon _{s} ds$ is as in the previous lemma. Young's continuity theorem, see Theorem \ref{cty-rough}, states that $x^\epsilon$ converges weakly
provided $X^\epsilon$ converges weakly in a H\"older space of regularity greater than $\f 1 2$.

For $H \in (\f 1 3, \f 1 2)$ we need to rewrite our equation into a rough differential equation,
$$\dot x_t^\epsilon =f(x_t^\epsilon) d\X_t^\epsilon,$$
where $\X^{\epsilon}$ is given by $X^{\epsilon}$ enhanced with its canonical lift
$$\XX^{\epsilon}_{s,t} = \int_s^t (X^\epsilon_r - X^\epsilon_s) dX^\epsilon_r.$$
As we restrict ourselves to one dimension we obtain, by symmetry,  $\XX^{\epsilon}_{s,t}= \f 1 2 {(X^{\epsilon}_{s,t})^2}$, hence, $\X^{\epsilon}$ converges to a fBm $\sigma B^H$ enhanced with $\mathbb{B}^H_{s,t}=\f 1 2  { \left( \sigma B^H_{s,t}\right)^2} $.
As the solution map, $\Phi$, to a RDE satisfies, see \cite{Friz-Hairer} or Theorem \ref{cty-rough} below,
$$|\Phi(\X^\epsilon)-\Phi(\mathbf{B}^H)|_{\C^{\gamma'}} \lesssim  \rho_{\gamma}(\X^\epsilon,  \mathbf{B}^H),$$
where $\rho_{\gamma}$ denotes the inhomogeneous rough path norm of regularity $\gamma$ and $\mathbf{B}^H = (\sigma B^H, \f 1 2  { \left( \sigma B^H_{s,t}\right)^2} )$. Thus, the 
$L^p$ convergence of the solutions follows from the $L^p$ convergence of the drivers, hence, we can  conclude the proof by Proposition \ref{th-example}.
\end{proof}

 \begin{remark}
 Krammer-Smoluchowski limits/Kinetic fBM's  are studied in
 \cite{Boufoussi-Ciprian, Zhang-08,Al-Talibi-Hilbert}. See also \cite{Fannjiang-Komorowski-2000, Friz-Gassiat-Lyons, Friz-Hairer}.
\end{remark}

\begin{remark}
	\label{remark-variance}
	\	
	\begin{enumerate}
		\item 
		For $H<\f 12$ and $m=1$, Theorem \ref{theorem-CLT} appears to contradict with Proposition \ref{th-example}; in the first we claim the limit is a Brownian motion, whereas  in the second we claim that it is  a fractional Brownian motion.  Both results are correct and can be easily explained.   It lies in the fact that $\int_\R \rho(s)\, ds$ vanishes if $H<\f 12$, and so the 
		Brownian motion limit is degenerate. 
		Since  according to   \cite{Cheridito-Kawaguchi-Maejima}, \begin{equation}\label{cor5}
		\rho(s) = \sigma^2 \frac{\Gamma(2H+1) \sin(\pi H)}{2 \pi} \int_{\R} e^{isx} \frac{\vert x \vert^{1-2H}}{1 + x^2} dx,
		\end{equation}
		and by the decay estimate from (\ref{cor1}), $\rho$ is integrable,
		$s(\lambda) $ is the value at zero of the inverse Fourier transform of $\rho(s)$, which is
		up to a multiplicative constant  $\frac{\vert \lambda \vert^{1-2H}}{1 + \lambda^2}$.
		This is also the spectral density of $y_t$ and has value $0$ at $0$. 
		This means we have scaled too much and the correct scaling is to multiply the integral $\int_0^{t} y^{\epsilon}_s ds$ by 
		$\epsilon^{H-1}$ in which case we obtain a fBm as  limit.
		\item For $m>1$ and $H<\f 12$ the Wiener limit is not trivial. Indeed,
		$$\begin{aligned}
		\int_\R  \rho(s)^m ds
		&=C \int_\R \stackrel{m} {\overbrace{ \int _\R \dots \int_{\R} }}   \prod_{k=1}^m e^{isx_k} \frac{\vert x_k \vert^{1-2H}}{1 + |x_k|^2} dx_1\dots dx_m\, ds\\
		&= C  \stackrel{m} {\overbrace{  \int _\R \dots \int_{\R}}} \frac{\vert x_2+\dots + x_m \vert^{1-2H}}{1 + |x_2+\dots + x_m|^2} \prod_{k=2}^m \, \frac{\vert x_k \vert^{1-2H}}{1 + |x_k|^2} \not =0.
		\end{aligned}$$
	\end{enumerate}
\end{remark} 

\subsection{The 1d fluctuation problem}
\label{homo}
In this section we give an application of Theorem A.
Given a  function $g \in L^1(\mu)$ we denote $ \bar{g}=\int_{\R} g(y) \mu(dy)$.

\begin{lemma}\label{lemma-pseudo-ergodic}
The stationary Ornstein-Uhlenbeck process is ergodic. Thus,  $\int_0^t g(y^{\epsilon}_s) ds\to t \bar g$  in probability for every $g \in L^1(\mu)$.
\end{lemma}
\begin{proof}
A stationary Gaussian process is ergodic if its spectral measure has no atom, 
see \cite{Cornfeld-Fomin-Sinai, Samorodnitsky}. The spectral measure $F$  of a stationary Gaussian process is obtained from  
 Fourier transforming  its correlation function and 
$\rho(\lambda)=\int_\R e^{i \lambda x} dF(x)$.
According to   \cite{Cheridito-Kawaguchi-Maejima}:
 \begin{equation}\label{cor5-2}
\rho(s) =  \frac{ \Gamma(2H+1) \sin(\pi H)}{2 \pi} \int_{\R} e^{isx} \frac{\vert x \vert^{1-2H}}{1 + x^2} dx,
\end{equation}
so the spectral measure is absolutely continuous with respect to the Lebesgue measure with spectral density, up to a non-zero  constant, given by $s(x) =  \f { \vert x \vert^{1-2H}}{1+ x^2}$. 
Since $\int_0^t g(y^{\epsilon}_t) dt$ equals $ \epsilon \int_0^{\f t \epsilon } g(y_s) ds$ in law,  the former converges in law to  $ t \bar{g}$ by Birkhoff's ergodic theorem. The claim now follows as weak convergence to a constant  implies  convergence in probability.
\end{proof}

In the following proof we will need the following theorem from Young/rough path theory for details we refer to \cite{Friz-Hairer,Friz-Victoir,Lyons94,Lyons-Caruana-Levy}. We denote the space of rough paths of  regularity $\beta$ by $\FC^{\beta}$.

\begin{theorem}\label{cty-rough}
	Let $Y_0 \in \R^d, \beta \in (\f1 3, 1), \, \gamma \in (\f 1 2, 1) \, f \in \C^3_b(\R^d,\R^d) $, $h \in \C_b^{2}(\R^d,\R^d)$, $\X \in \FC^{\beta}([0,T],\R)$ and $\mathbf{Z} \in \FC^{\gamma}([0,T],\R)$ such that $\beta + \gamma >1$. Then, the differential equation
	\begin{equation}\label{example-sde}
	Y_t = Y_0 + \int_0^t f(Y_s) d\mathbf{X}_s + \int_0^t h(Y_s) d\mathbf{Z}_s
	\end{equation}
	has a unique solution which belongs to $\C^{\beta \wedge \gamma}$. Furthermore, the solution map  $\Phi_{f,h}: ~\R\times \FC^{\beta}([0,T], \R) \times \FC^{\gamma}([0,T], \R)
	\to  \C^{\beta \wedge \gamma}([0,T],\R)$, where the first component is the initial condition and the second and third components the drivers, is continuous.
\end{theorem}
Given a centred function $G \in L^2(\mu) $, with chaos expansion $G=\sum_{k=m}^{\infty} c_k H_k$ and  let $c \geq 0$ be given by 
\begin{equation}\label{c-square}
c^2=\left\{ \begin{array}{ll}  (\f {c_m m!}{K(H,m)})^2, \quad & H^*(m)>\f 12\\
2\sum_{k=m}^\infty (c_k)^2 k! \int_0^\infty \rho^k(s) ds, \quad &H^*(m)<\f 12\\
2m!(c_m)^2, \quad &H^*(m)=\f 12.
\end{array} \right.
\end{equation}

\begin{theorem} 
Let $H \in (\f 1 3,1) \setminus \{ \f 1 2 \}$,  $f\in \C_b^3(\R^d, \R^d)$, $h \in \C^{2}_b(\R^d;\R^d)$, $G \in \C(\R,\R)$ satisfying Assumption \ref{assumption-single-scale-not-continuous} and $g \in \C_b(\R;\R)$.
Let  $\alpha(\epsilon) = \alpha(\epsilon,H^*(m))$. Fix a finite time $T$ and consider
\begin{equation}\label{limit-eq}
\dot x_t^\epsilon =\alpha(\epsilon) f(x_t^\epsilon) G(y^{\epsilon}_t)+ h(x^{\epsilon}_t)g(y^{\epsilon}_t), 
  \qquad  x_0^\epsilon=x_0.
 \end{equation} 
\begin{enumerate}
\item If
$H^*(m) > \f 1 2$, $x_t^\epsilon$ converges weakly in $\C^{\gamma}([0,T],\R^d)$  to the solution to the Young differential equation
$d\bar x_t = c f(\bar x_t) \,dZ_t^{H^*(m),m} +\bar{g} h(\bar{x}_t) dt$  with initial value $x_0$ for $\gamma \in (0, H^*(m)- \f 1 p)$.
\item If
$H^*(m) \leq \f 1 2$, 
 $x_t^\epsilon$ converges weakly  in  $\C^\gamma([0,T],\R)$ to the solution of the Stratonovich stochastic differential equation
$d\bar x_t = c  f(\bar x_t) \circ \,dW_t + \bar{g} h(\bar{x}_t) dt$ with $ \bar x_0=x_0$, where $\gamma \in(0, \f 1 2- \f 1 p)$.
\end{enumerate} 
\end{theorem}
\begin{proof}  \label{proof-theorem-C}
As in Proposition \ref{prop-linear-driver} we can rewrite our equations as Young/rough differential equations and therefore reduce our analysis to the drivers $\left(\alpha(\epsilon) \int_0^t G(y^{\epsilon}_s) ds, \int_0^t g(y^{\epsilon}_s) ds\right)$.
By Theorem \ref{theorem-CLT}, $\alpha(\epsilon) \int_0^t G(y^{\epsilon}_s) ds$ converges in finite dimensional distributions either to a Wiener or a Hermite process. By Lemma \ref{lemma-pseudo-ergodic}, $\int_0^t g(y^{\epsilon}_s) ds$ converges in probability  to the deterministic path  $t \bar{g}$. Hence, 
$\left(\alpha(\epsilon) \int_0^t G(y^{\epsilon}_s) ds, \int_0^t g(y^{\epsilon}_s) ds\right)$ converges jointly in finite dimensional distributions.
Furthermore,  $\Vert\int_0^t g(y^{\epsilon}_s) ds \Vert_{\infty} \leq t \Vert g \Vert_{\infty}$, this combined with the moment bounds obtained in Theorem \ref{theorem-CLT} enables us to apply Theorem \ref{cty-rough} to conclude the proof. 
\end{proof}

\begin{remark}
		The constant $c$ could be $0$, for further details see Remark \ref{remark-variance}.
\end{remark}

\section{Proof of Theorem A}

We first establish the  $L^2(\Omega)$ convergence of  $X_t^\epsilon= \alpha(\epsilon)\int_0^{t} G(y^{\epsilon}_s)ds$,
where $G=\sum_{k=m}^\infty c_k H_k$ has low Hermite rank, followed by  a reduction theorem. We then prove moment bounds and  conclude the proof of Theorem~\ref{theorem-CLT}.

\subsection{Preliminary lemmas} 
The basic scalar valued functional limit theorem, for low rank Hermite functions,  was proved in \cite{Taqqu} for
$\epsilon^{H^*(m)} \int_0^{\f t \epsilon} G(X_s) ds$ with  $X_t = \int_{\R} p(t-\xi) dW_{\xi}$  a moving average,
where in order to prove  convergence one uses the self-similarity of a Wiener process leading to weak convergence as this equivalence of course is only in law. 
Nevertheless, in our case we can choose a properly scaled fast variable and write, $y_t^{\epsilon} = \int_{\R} \hat{g} (\f {t-\xi} {\epsilon}) dW_{\xi}$ for a function $\hat g$, and thus avoid using self-similarity. 
The key idea is to write a Wiener integral representation beginning with
	\begin{equation}\label{y-integral}
	\begin{aligned}
	y_t^\epsilon&=\epsilon^{-H} \sigma \int_{-\infty}^t e^{-\f {t-r} \epsilon} dB^H_r
	=\int_{\R} h_\epsilon(t,s) dW_s, \quad \hbox{ where,}\\
	h_\epsilon(t,s)&= \epsilon^{-\f 12} 
	 \f {\sigma} {c_1(H)} e^{-\f {t-s} \epsilon}  \int_0^{\f {t-s} \epsilon}  e^v \;v_{+}^{H-\f 32}\; dv,
	\end{aligned}\end{equation}
and  $c_1(H) = \sqrt{ \int_{-\infty}^0 \left( (1-s)^{H-\f 1 2} - (-s)^{H- \f 1 2} \right)^2 ds + \f 1 {2H} }$.
This can be obtained  by applying the integral representation for fBM's:
\begin{equation}\label{fbm-i}\begin{aligned}
B_t^H=\int_{-\infty}^{\infty} g(t,s)dW_s,  \qquad  \hbox{ where } \quad
g(t,s)= \f {1} {c_1(H)} \int_{0}^t  (r-s)_+^{H-\f 32} dr,\end{aligned}\end{equation} and 
by repeated applications of integration by parts (to the Young integrals):
\begin{align*}
&\sigma \int_{-\infty}^t e^{- \f {t-s} {\epsilon}} dB^H_s = \sigma B^H_t - \f \sigma \epsilon \int_{-\infty}^t e^{- \f {t-s} {\epsilon}} B^H_s ds\\
&= \sigma  B_t^H - \f \sigma \epsilon \int_{-\infty}^t  e^{- \f {t-s} {\epsilon}} \left(  \int_{\R} g(s,r) dW_r \right)ds
=\sigma  \int_{\R} \int_{-\infty}^t e^{- \f {t-s} {\epsilon}} \partial_s g(s,r) ds dW_r\\ 
&= \f{\sigma} {c_1(H)} \int_{\R} \int_{-\infty}^t e^{- \f {t-s} {\epsilon}} (s-r)_{+}^{H- \f 3 2} ds dW_r.
\end{align*}
One may also use the following, see \cite{Pipiras-Taqqu}, taking $f\in L^1\cap L^2$:
$$
\int_{\R} f(u) dB^H_u = \f {1} {c_1(H)}\int_{\R}  \int_{\R} f(u) (u-s)_{+}^{H- \f 3 2} \,du  \,dW_s.$$

\begin{lemma}\label{kernel lemma}  Let $\lambda$ denote the Lebesgue measure, then  as $\epsilon\to 0$, 
$\epsilon^{H^*(m) -1} \int_0^t H_m(y^{\epsilon}_s) ds$ converges to $ \f { m!} { K(H,m)} Z_t^{H^*(m),m}$ in $L^2(\Omega)$. 
Equivalently, 
	\begin{equation}
\left \Vert \int_0^t
	\prod_{i=1}^m h_\epsilon(s, u_i)  ds -   \int_0^t \prod_{i=1}^m (s- u_i)_+^{H-\f 32}ds\right \Vert_{L^2(\R^m, \lambda)} \to 0. 
	\end{equation}
\end{lemma}
\begin{proof}
This can be shown  by applying  \cite[Theorem 4.7]{Taqqu}, where weak convergence is obtained. With a small modification and using \cite[Lemma 4.5, Lemma 4.6]{Taqqu} directly we obtain the $L^2(\Omega)$ convergence:
$$ \E \left(  \int_{\R^m}  \int_0^t \prod_{i=1}^m p\left( \f {s- \xi} {\epsilon}\right) \epsilon^{H- \f 3 2} ds dW_{\xi} - \f{Z^{H^*(m),m}} {K(H,m)} \right)^{ 2 }\to 0,$$
using the Wiener integral representation of the Hermite processes this is equivalent, by a multiple Wiener-It\^o isometry, to 
\begin{equation}\label{explicit-kernel-2}
\int_{\R^m} \left( \int_0^t \prod_{i=1}^m p\left( \f {s- \xi_i} {\epsilon}\right) \epsilon^{H - \f 3 2} ds - \int_0^t \prod_{i=1}^m \left( s- \xi_i \right)_{+}^{H- \f 3 2} ds \right)^2 d\xi_1 \dots d\xi_m \to 0.
\end{equation} 
Examining Taqqu's proof,  we note that in fact the $L^2$ convergence of (\ref{explicit-kernel-2})
is obtained under the following conditions.
\begin{itemize}
	\item[{1}] $\int_{\R} p(s)^2 ds < \infty$.
	\item[{2}] $ \vert p(s) \vert \leq C s^{H- \f 3 2} L(u)$ for almost all $s>0$.
	\item[{3}] $p(s) \sim s^{H - \f 3 2 } L(s)$ as $s \to \infty$.
	\item[{4}] There exists a constant $\gamma$ such that $0<\gamma< (1-H)\wedge  (H- (1-\f 1 {2m}))$ such that $\int_{-\infty}^0 \vert p(s) p(xy+s)\vert ds = o(x^{2H-2} L^2(x)) y^{2H-2-2\gamma}$ as $x \to \infty$ uniformly in $y \in (0,t]$.
\end{itemize}
where $L$ denotes a slowly varying function (for every $\lambda >0$ $\lim_{x \to \infty} \frac{L(\lambda x)}{L(x)} ) =1$).
Set \begin{equation}\label{g} \hat{g}(s)=  \f {\sigma} {c_1(H)}  e^{-s} \int_0^s e^u u_{+}^{H- \f 3 2} du,
\end{equation}
then,
 \begin{align*}
y^{\epsilon}_t
=     \epsilon^{-\f 1 2}\int_{- \infty}^{t} \hat{g}\left( \f{t-s} {\epsilon} \right) dW_s.
\end{align*}
We are now in Taqqu's framework and it is only left the check $\hat{g}$ defined by (\ref{g})  satisfies these conditions. To increase readability we suppress the constant $\f {\sigma} {c_1(H)}$ in the following computations.
For  $s<1$,
\begin{align*}
e^{-s} \int_0^s e^u u^{H- \f 3 2 } du  &\leq   \int_0^s u^{H- \f 3 2 } du
\lesssim  s^{H - \f 1 2}.
\end{align*}
We calculate for $s>1$ via integration by parts
\begin{align*}
e^{-s} \int_0^s e^u u^{H- \f 3 2 } du  &\leq 
e^{-s} \int_0^1 e^u u^{H- \f 3 2 } du + e^{-s} \int_1^s e^u u^{H- \f 3 2 } du \\
&\lesssim  e^{-s} +s^{H- \f 3 2} +  e^{-s} \int_1^s e^u u^{H- \f 5 2 } du
\lesssim   s^{H- \f 3 2}.
\end{align*}
This of course implies that $\hat{g}$ is $L^2(\lambda)$ integrable.
Finally observe that $\int_{-\infty}^0 \vert \hat{g}(s) \hat{g}(xy+s)\vert ds = 0$ as $\hat{g}(s)=0$ for $s<0$.
With these we apply  \cite[Theorem 4.7]{Taqqu} to conclude the $L^2$ convergence of the kernels.
\end{proof}

\begin{lemma}\label{L^2-kernel}
Let  $G\in L^2(\mu)$ be a centred function with Hermite rank $m$ satisfying  $H^*(m)> \f 1 2$. Let $H\in ( \f 1 2 ,1)$.
Then the following statements hold  for  the  stationary scaled fOU process $y^{\epsilon}_s$. Fix $t>0$, then,
$$\left\|  \epsilon^{H^*(m)-1}  \int_{0}^{t} G(y^{\epsilon}_s) ds- \f {c_m m!} { K(H^*(m),m)} Z^{H^*(m),m}_t\right\|_{L^2(\Omega)} \to 0.$$
\end{lemma}

\begin{proof}
For $G= H_m$ the claim has already been shown in Lemma \ref{kernel lemma}.
To conclude the claim in the case of a general $G$, we compute,
\begin{equation}\label{reduction_Hermite}
\begin{split}
&\left\| \epsilon^{H^*(m)-1}  \int_{0}^{t} (G-c_m H_m)(y^{\epsilon}_s) ds\right\|^2_{L^2(\Omega)}
=   \epsilon^{2H^*(m)-2} \sum_{k=m+1}^{\infty} c_k^2 k!  o(\epsilon^{2H^*(m)-2})  
\to 0
\end{split}
\end{equation}
as $\sum_{k=m+1}^{\infty} c_k^2 \sqrt{k!} < \infty$ as $G \in L^2(\mu)$.
This finishes the proof.
\end{proof}

\medskip  

The fact that only the first term in the chaos expansion gives a contributions is in the literature often called a reduction lemma. In the high Hermite rank case however it is not possible to restrict one's analysis to a pure Hermite polynomial, but as the next lemma shows finite linear combinations are indeed sufficient. To make the application later on easier we directly prove it in the multi-dimensional case.

\begin{lemma}[Reduction Lemma]\label{reduction}
	Fix $H \in (0,1)\setminus\{\f 12\}$. 
	For  $M\in \N$, define the truncated functions:
	$$G_{k,M} = \sum_{j=m_k}^{M} c_{k,j} H_j, \qquad X^{k,\epsilon}_{M}(t)= \alpha_k(\epsilon) \int_0^{t} G_{k,M}(y^{\epsilon}_s)ds.$$
	If for every $M \in \N$, 
	$(X^{1,\epsilon}_{M}, \dots, X^{N,\epsilon}_{M})\stackrel{(\epsilon\to 0)} \longrightarrow  (X^{1}_{M}, \dots, X^{N}_{M})$ 
	in finite dimensional distributions, then, 
	$$\left(X^{1,\epsilon}, \dots, X^{N,\epsilon}\right) \stackrel{(\epsilon\to 0)} \longrightarrow  (X^{1}, \dots, X^{N})$$
	in finite dimensional distributions.
\end{lemma} 
\begin{proof}
Firstly, 
	\begin{align*}
	X^{k,\epsilon}(t) -  X^{k,\epsilon}_{M}(t) &=  \alpha_k(\epsilon) \int_0^{t} \Big( G_k(y^{\epsilon}_s)- G_{k,M} (y^{\epsilon}_s)\Big)ds
	=  \alpha_k(\epsilon) \int_0^{t} \sum_{j=M+1}^\infty c_{k,j} H_j(y^{\epsilon}_s)ds.
	\end{align*}
	Using properties of the Hermite polynomials we obtain
	\begin{align*}
	&\E \left(   \alpha_k(\epsilon) \int_0^{t} \sum_{j=M+1}^\infty c_{k,j} H_j(y^{\epsilon}_s)ds \right)^2 
	=  \alpha_k(\epsilon)^2 \int_0^{t} \int_0^{t} \sum_{j=M+1}^\infty (c_{k,j})^2 \E \left(  H_j(y^{\epsilon}_s) H_j(y^{\epsilon}_r) \right) dr ds\\
	&=   \alpha_k(\epsilon)^2 \sum_{j=M+1}^\infty (c_{k,j})^2 j! \int_0^{t} \int_0^{t}  \rho^{\epsilon}(\vert s-r \vert)^j dr ds
	\lesssim   \,\sum_{j=M+1}^\infty (c_{k,j})^2 j!
	\end{align*}
	As $\sum_{j=m}^\infty (c_{k,j})^2 j! < \infty$ we obtain $ \sum_{j=M+1}^\infty (c_{k,j})^2 j!\to 0$ as  $M \to \infty$.  Thus,
	\begin{equation}\label{Bill}
	\lim_{M \to \infty}\lim_{\epsilon \to 0}\E \left(   \alpha_k(\epsilon) \int_0^{t} G_k(y^{\epsilon}_s)ds -  \alpha_k(\epsilon) \int_0^{t} G_{k,M}(y^{\epsilon}_s)ds \right)^2 \to 0,
	\end{equation}
Let $ \{t_{\gamma_{k,l}}, k\le N, l\le A \} $ be a sequence of positive numbers.
Now, by the triangle inequality,
	$$\lim_{M \to \infty}\lim_{\epsilon \to 0}
	\left\Vert \sum_{k,l} \gamma_{k,l} \left( X^{k,\epsilon} (t_l) - X^{k,\epsilon}_M  (t_l)  \right) \right\Vert_{L^2(\Omega)} \to 0.$$
With Theorem 3.2 in \cite{Billingsley} this proves the claim.
\end{proof}

\subsection{Moment bounds}
\label{pre}

We will use some results from Malliavin Calculus. Let $x_s$ be a stationary Gaussian process with $\beta(s) = \E \left( x_s x_0 \right)$, such that $\beta(0)=1$. As a real separable Hilbert space we use  $\mathscr{H} = L^2(\R_+,\nu)$ where for a Borel-set $A$ we have $\nu(A) = \int_{\R_+} \1_{A} (s) d\beta_s$. Let $\H^{\otimes q}$ denote the $q$-th tensor product of $\H$. For $h\in \H$, we may define the Wiener integrals $W(h)=\int_0^\infty h_s dx_s$ by $W([a,b])=x(b)-x(a)$ (where $a, b\ge 0$), linearity and the Wiener isometry 
($\<\1_{[0,t]}, \1_{[0,s]}\>=\beta(t-s)$).
Iterated Wiener integrals are defined similarly and by its values on indictor functions:
$I_m(\1_{A_1\times \dots \times A_m})=\prod_{i=1}^m W (A_i)$ where $A_i$ are pairwise disjoint Borel subsets of $\R_+$.
If $\F$ denotes the $\sigma$-field generated by $x$, 
then any $\F$-measurable $L^2(\Omega)$ function $F$ has the chaos expansion:
$F=\E F+ \sum_{m=1}^\infty I_m(f_m)$ where $f_m\in L^2( \R_+^m)$. This is due to the fact that $L^2(\Omega)=\bigoplus_{m=0}^{\infty} \H_m$ where $\H_m$ is the closed linear space generated by $\{ H_m(W(h)): \Vert h \Vert_{L^2}=1\}$, $H_m$ are the $m$-th Hermite polynomials, and  $\H_m=I_m(L_{\hbox{sym}}^2(\R_{+}^m))$.
The last fact is due to $H_m(W(h))= I_m( h^{\otimes ^{m} })$. In the following
${\mathbb{D}^{k,p}(\mathscr{H}^{\otimes m})}$ denotes the closure of Malliavin smooth random variables under the following norm
$\Vert u \Vert_{\mathbb{D}^{k,p}(\mathscr{H}^{\otimes m})} = \left( \sum_{j=0}^{k} \E \left( \Vert D^j u \Vert_{\mathscr{H}^{\otimes m}}^p  \right) \right)^{\f 1 p}$.

\begin{lemma}[Meyer's inequality] \cite{Nourdin-Peccati} 
	\label{Meyer}
	Let $\delta$ denote the  divergence operator.
	  Then  for $u \in \mathbb{D}^{k,p}(\mathscr{H}^{\otimes m})$,
	\begin{equation}
	\Vert \delta^m(u) \Vert_{L^p(\Omega)} \lesssim  \sum_{k=0}^m  \Vert u \Vert_{\mathbb{D}^{k,p}(\mathscr{H}^{\otimes m})}.
	\end{equation}
\end{lemma}

\begin{lemma}\label{representation}
	\cite{Campese-Nourdin-Nualart}
	If $G:\R\to \R$ is a function of Hermite rank $m$, then $G$ has the following  multiple Wiener-It\^o-integral representation:
	\begin{equation}\label{Lp-eq}
	G(x_s) = \delta^m\left( G_m(x_s) \1_{[0,s]}^{\otimes m}\right),
	\end{equation}
	where $G_m$ has the following properties: \begin{itemize}
		\item [(1)]   $\Vert G_m \Vert_{L^p(\mu)} \lesssim \Vert G \Vert_{L^p(\mu)}$,
		\item  [(2)]
		$G_m(x_1)$ is $m$ times Malliavin differentiable and its $k^{th}$ derivative, denoted by  $G_m^{(k)}(x_1)\1_{[0,1]}^{\otimes k}$,
		satisfies $ \Vert G_m^{(k)} \Vert_{L^p(\mu)} \lesssim  \Vert G \Vert_{L^p(\mu)}$.
	\end{itemize}
\end{lemma}
In the lemma below we estimate the moments of $\int_0^{t} G(x_{\f r \epsilon})  dr$, where we need the multiple Wiener-It\^o-integral representation above to transfer the
correlation function to $L^2$ norms of indicator functions. We use an idea from \cite{Campese-Nourdin-Nualart} for the
estimates below. 

\begin{lemma}\label{Lp-bounds} 
	Let  $x_t=W([0,t])$ be a stationary Gaussian process with correlation $\beta(t)=\big| \E(x_t x_0)\big|$, stationary distribution $\mu$ and $\H$ the $L^2$ space over $\R_{+}$ with measure $\beta(r) dr$. 	If $G$ is a function of Hermite rank $m$ and $ G \in L^p(\mu)$, for $p>2$, then,
	\begin{equation}\label{le5.4-1}
	{\begin{split}
		\left\Vert  \f {1} {\epsilon} \int_0^{t } G(x_{\f r \epsilon})  dr \right\Vert_{L^p(\Omega)} 
		&\lesssim \Vert G \Vert_{L^p(\mu)} \left ( \int_0^{\f t \epsilon} \int_0^{\f t \epsilon} \beta(\vert u -r \vert)^m dr du \right)^{\f 1 2}.\end{split}}
	\end{equation}
	For the stationary scaled fractional OU process $y^{\epsilon}_t$, we have
	\begin{equation}\label{le5.4-2}
	{\begin{split}
		\left\Vert  \f {1} {\epsilon} \int_0^{t } G( {y_r^\epsilon})  dr \right\Vert_{{L^p(\Omega)}} 
		&\lesssim \left\{	\begin{array}{lc}
		\Vert G \Vert_{L^p(\mu)}\; \sqrt {\f t \epsilon  \int_0^\infty \big|\rho^m(s) \big|ds } ,  \quad  &\hbox {if} \quad H^*(m)<\f 12,\\
		\Vert G \Vert_{L^p(\mu)} \; \sqrt {\f t \epsilon \ln| \f {1} \epsilon|}, \quad  &\hbox {if} \quad H^*(m)=\f 12,\\
		\Vert G \Vert_{L^p(\mu)} \; \left(  \f t \epsilon\right) ^{H^*(m)},  \quad &\hbox { otherwise.}
		\end{array} \right.
		\end{split}},
	\end{equation}
	in particular,
	\begin{equation}
	\left\Vert \int_0^{t } G(y^{\epsilon}_r)  dr \right\Vert_{L^p(\Omega)} \lesssim  \f {\Vert G \Vert_{L^p(\mu)} t^{H^*(m) \vee \f 1 2}} {\alpha(\epsilon,H^*(m))}.
	\end{equation}
	
\end{lemma}

\begin{proof}
	We first use Lemma \ref{representation}  and then apply Meyer's inequality from Lemma \ref{Meyer} to obtain
	\begin{align*}
	&\left\Vert \f {1} {\epsilon} \int_0^{t } G(x_{\f r \epsilon})  dr \right\Vert_{L^p(\Omega)}   =\left\Vert \int_0^{\f t \epsilon} G(x_r)  dr \right\Vert_{L^p(\Omega)}   
	=\left \Vert \int_0^{\f t \epsilon}\delta^m\left( G_m(x_r) \1_{[0,r]}^{\otimes m}\right)\, dr\right \Vert_{L^p(\Omega)}  \\
	&\lesssim  \sum_{k=0}^m \left \Vert \int_0^{\f t \epsilon}  D^k \left(G_m(x_r) \1_{[0,r]}^{\otimes m}\right)  dr 
	\right\Vert_{L^{p}(\Omega,\mathscr{H}^{\otimes  m +k})}
	= \sum_{k=0}^m\left \Vert \int_0^{\f t \epsilon}  G_m^{(k)} (x_r) \1_{[0,r]}^{\otimes {m +k}}  dr 
	\right\Vert_{L^{p}(\Omega,\mathscr{H}^{\otimes  m +k})}.
	\end{align*}
Here for $G_m$ is as given in Lemma \ref{representation} and $G_m^{(k)}$ denotes its k-th order Malliavin derivative. 
	We estimate the individual terms using the linearity of the inner product and the isometry $\<  \1_{[0,r]},  \1_{[0,s]}\>_\H=\E(x_rx_s)=\beta(r-s)$,
	\begin{align*} &\left( \left \Vert \int_0^{\f t \epsilon}  G_m^{(k)} (x_r) \1_{[0,r]}^{\otimes {m +k}}  dr 
	\right\Vert_{\mathscr{H}^{\otimes  m +k}}\right)^2
	=  \left\<  \int_0^{\f t \epsilon}  G_m^{(k)} (x_r) \1_{[0,r]}^{\otimes {m +k}}  dr ,  \int_0^{\f t \epsilon}  G_m^{(k)} (x_u) \1_{[0,r]}^{\otimes {m +k}}  du 
	\right\>_{\H^{\otimes  m +k}}\\
	& = \int_0^{\f t \epsilon} \int_0^{\f t \epsilon}    G_m^{(k)} (x_r) G_m^{(k)} (x_u) \<\1_{[0,r]}^{\otimes {m +k}} , \1_{[0,u]}^{\otimes {m +k}} \>_{\H^{\otimes  m +k}}\, dr du\\
	& = \int_0^{\f t \epsilon} \int_0^{\f t \epsilon}    G_m^{(k)} (x_r) G_m^{(k)} (x_u) \Big(  \beta(r-u)\Big) ^{m+k}  \, dr\,du.
	\end{align*}
	Using  Minkowski's inequality  we obtain
	\begin{align*}
	&\sum_{k=0}^m\left \Vert \int_0^{\f t \epsilon}  G_m^{(k)} (x_r) \1_{[0,r]}^{\otimes {m +k}}  dr 
	\right\Vert_{L^{p}(\Omega,\mathscr{H}^{\otimes  m +k})}\\
	&\leq  \sum_{k=0}^m \left( \left\Vert
	\int_0^{\f t \epsilon} \int_0^{\f t \epsilon}    G_m^{(k)} (x_r) G_m^{(k)} (x_u) \beta(r-u)^{m+k} drdu
	\right\Vert_{L^{\f p 2}(\Omega)}\right)^{\f 1 2}\\
	&\leq  \sum_{k=0}^m \left( 
	\int_0^{\f t \epsilon} \int_0^{\f t \epsilon}   \left\Vert G_m^{(k)} (x_r) G_m^{(k)} (x_u)  \right\Vert_{L^{\f p 2}(\Omega) } \beta(r-u)^{m+k} drdu
	\right)^{\f 1 2}.
	\end{align*}
	We then  estimate $\E |G_m^{(k)} (x_r) G_m^{(k)} (x_u)|^{\f p 2}$ by H\"older's inequality and use the fact that $x_t$ is stationary. Since the right hand side is then controlled by
	\begin{align*}
	RHS & \leq  \sum_{k=0}^m \Vert G_m^{(k)} \Vert_{L^p(\mu)} 
	\left( \int_0^{\f t \epsilon}\int_0^{\f t \epsilon}\vert \beta(\vert u-r \vert) \vert^{m+k}  dr du\right)^{\f 1 2}
	\lesssim  \Vert G \Vert_{L^p(\mu)} \left( \int_0^{\f t \epsilon}\int_0^{\f t \epsilon} \vert \beta(\vert u-r \vert) \vert^{m}  dr du \right)^{\f 1 2},
	\end{align*}
		concluding (\ref{le5.4-1}).
	We  finally apply Lemma \ref{Integrals} to conclude (\ref{le5.4-2}).
\end{proof}

Now we are ready to prove our main theorem.

\subsection{Concluding the proof}

\begin{proof}
\
\textbf{Step 1, CLT in the pure Wiener case}\\
We first deal with the high Hermite rank components.
For $k \leq n$ we define the truncated functions $G_{k,M}=\sum_{j=m_k}^{M} c_{k,j} H_j$ and set
$$
	X^{k, \epsilon}_M= \alpha_k(\epsilon) \int_0^{t} G_{k,M}(y^{\epsilon}_s) ds.
	$$
	Then, by  the reduction Lemma \ref{reduction} above,  it is sufficient to show the convergence of $(X^{1,\epsilon}_M, \dots , X^{n,\epsilon}_M)$ for every $M$. By  \cite{BenHariz} and \cite{Buchmann-Ngai-bordercase} each component alone converges to a Wiener process. Hence, as each $X^{k,\epsilon}_M$ belongs to a finite chaos we can make use of the normal approximation theorem from \cite[Theorem 6.2.3]{Nourdin-Peccati}: if each component of a family of mean zero vector valued stochastic processes, with components of the form $I_{q_i}(f_{i,n}) $, where $f_{i,n}$ are symmetric $L^2$ functions in $q_i$ variables,  converges in law to a Gaussian process, then they converge  jointly in law to a vector valued Gaussian process, provided that their covariance functions converge. Furthermore, the covariance functions of the limit distribution are  $\lim_{\epsilon \to 0} \E [ X^{i,\epsilon}(t) X^{j,\epsilon}(s)] $.
	Let $m= \min(m_i,m_j)$  we use $$\E (H_k(y^{\epsilon}_t)H_l(y^{\epsilon}_s))=\delta_{k,l} \left(\E(y^{\epsilon}_sy^{\epsilon}_t)\right)^k$$ to obtain, for $s \leq  t$, 
	$$\begin{aligned} &\E\left[ \alpha_i(\epsilon) \alpha_j(\epsilon)\int_0^{t}G_{i,M}(y^{\epsilon}_u)du  \int_0^{s} G_{j,M}(y^{\epsilon}_r) dr \right] \\
	&=  \sum_{k=m}^M \alpha_i(\epsilon) \alpha_j(\epsilon) c_{i,k} c_{j,k} (k!)^2 \int_0^{t}\int_0^{s} (\E(y^{\epsilon}_r y^{\epsilon}_u))^k  dr du\\
	&=  \sum_{k=m}^M  \alpha_i(\epsilon) \alpha_j(\epsilon) c_{i,k} c_{j,k} (k!)^2 \left( \int_0^{s} \int_0^{s} \rho^{\epsilon}(u-r)^k dr du + \int_{s }^{t} \int_0^{s} \rho^{\epsilon}(u-r)^k  dr du \right).
	\end{aligned}$$
	By Lemma \ref{Integrals} we obtain,  for $\epsilon \to 0$,
	\begin{align*}
	\alpha_i(\epsilon) \alpha_j(\epsilon) \int_{s }^{t} \int_0^{s} \rho^{\epsilon}(u-r)^k dr du \to 0.
	\end{align*} 
	Hence,
	$$\begin{aligned}
	\lim_{\epsilon \to 0} RHS &= 2 \sum_{k=m}^M c_{i,k} c_{j,k} (k!)^2  \lim_{\epsilon \to 0} \left( \epsilon \alpha_i(\epsilon) \alpha_j(\epsilon) s \int_0^{\f s {\epsilon}} (\rho(v))^k  dv\right)  \\
	&= 2 \left( s \wedge t \right)  \sum_{k=m}^M c_{i,k} c_{j,k} (k!)^2 \int_0^\infty \rho(u)^k du\\
	&= 2 \left( s \wedge t \right) \,\int_0^\infty \E (G_{i,M}(y_s) G_{j,M}(y_0) )ds
	,\end{aligned}$$
	proving the finite chaos case.	 
	We now prove that the correlations of the limit converge as $ M \to \infty$. Indeed,
	\begin{align*}
	\lim_{M \to \infty}  2 \left( s \wedge t \right)  \sum_{k=m}^M c_{i,k} c_{j,k} (k!)^2 \int_0^\infty \rho(u)^k du
	&=  2 \left( s \wedge t \right) \sum_{k=m}^{\infty} c_{i,k} c_{j,k} (k!)^2 \int_0^\infty \rho(u)^k du\\
	&=  2  \left( s \wedge t \right) \,\int_0^\infty \E (G_i(y_s) G_j(y_0) )ds.
	\end{align*}
	As $G_{i,M} \to G_i$ in $L^2(\mu)$, and similarly for $j$, this proves the joint convergence of the high Hermite rank components.

\textbf{Step 2, CLT in the pure Hermite case}\\
In this step we focus on the vector component whose entries satisfy $H^*(m_k)> \f 1 2$. Recall,
this implies ${H> \f 1 2}$. 
By Lemma \ref{L^2-kernel}, evaluations of each component of $(X^{n+1}, \dots, X^N)$ converge in $L^2(\Omega)$. Hence, they converge as well jointly in $L^2(\Omega)$.  Now, choose finitely many points $t_{k,j} \in [0,T]$ and constants $ a_{k,j} \in \R$, then, $\sum_{k,j} a_{k,j} X^{k,\epsilon}_{t_{k,j}}$ converges in $L^2(\Omega)$ to $\sum_{k,j} a_{k,j} X^{k}_{t_{k,j}}$ and thus we may conclude joint convergence in finite dimensional distributions by an application of the Cramer-Wold theorem.

\textbf{Step 3, Joint convergence}\\
We have already shown that $X^{W,\epsilon} \to X^W$ and $X^{Z,\epsilon} \to X^Z$ in finite dimensional distributions, it is only left to prove their joint convergence.
By  Lemma \ref{reduction} and Equation (\ref{reduction_Hermite}) we may again reduce the problem to 
	$$G_i= \sum_{k=m_i}^M c_{i,k} H_k, \quad   G_j=  c_{j,m_j} H_{m_j},  \qquad 1 \leq i \leq n,\; j>n.$$
	Now, we can rewrite $H_m(y^{\epsilon}_s)= I_m(f^{m,\epsilon}_s)$, where $I_m$ denotes a $m$-fold Wiener-It\^o integral and a function $f^{m,\epsilon}_s \in L^2(\R^m,\mu)$. 
	Now, for $1 \leq i \leq n$ we obtain,
	\begin{align*}
	\alpha_i(\epsilon) \int_0^t G_i(y^{\epsilon}_s)ds &= \alpha_i(\epsilon) \int_0^t \sum_{k=m_i}^{M} c_{i,k} H_k(y^{\epsilon}_s) ds
	\\
	&=  \alpha_i(\epsilon)\int_0^t \sum_{k=m_i}^{M} c_{i,k} I_k(f^{k,\epsilon}_s) ds= \sum_{k=m_i}^{M} c_{i,k} I_k(\hat {f}^{k,\epsilon}_t),
	\end{align*}
	where  $$\hat f^{k,\epsilon}_t = \int_0^t f^{k,\epsilon}_s ds.$$ 
	 Similarly for $j>n$,   
	\begin{align*}
	\int_0^t G_j(y^{\epsilon}_s)ds &= \int_0^t  c_{j,m_j} H_{m_j}(y^{\epsilon}_s) ds
	= c_{j,m_j} I_{m_j}(\hat {f}^{m_j,\epsilon}_t).
	\end{align*}
	Hence, we only need to show that the collection of stochastic processes of the form $I_{m_k}(\hat {f}_t^{m_k,\epsilon})$ converges jointly in finite dimensional distribution.
	It is thus sufficient to show that for every finite collection of times, $t_{1}, \dots, t_Q \in [0,T]$, 
	 the vector,
	$\left\{I_{k}(\hat {f}^{k,\epsilon}_{t_{l}}), k=m, \dots, M, l =1,\dots, Q  \right\}$	converges jointly, where $m = \min_{k=1, \dots, N} m_k$.
	Let $n_0$ denote the smallest natural number such that $H^*(n_0) < \f 1 2$. For $k > n_0$, the collection  $I_k(\hat {f}^{k,\epsilon}_{t_{l}})$  converges to a normal distribution, hence, by the moment bounds in Lemma \ref{Lp-bounds}
	$ \|(k!)  I_k(\hat {f}^{k,\epsilon}_{t_{l}})\|_{H^{\otimes k}}= \sqrt{ \E  \left( I_{k}(\hat {f}^{k,\epsilon}_{t_l})^2\right) } $ converges to a constant.
	The convergence of such a sequence is equivalent to the convergence of the contractions, as defined in Equation (\ref{def-contraction}), of their kernels. Indeed,  by a generalised fourth moment theorem \cite[Theorem 1]{Nualart-Peccati},
	we have 
	$$
	\Vert \hat {f}^{k,\epsilon}_{t_{l}} \otimes_r \hat  {f}^{k,\epsilon}_{t_{l}} \Vert_{\H^{2k-2r}} \to 0, \qquad  r =1 ,\dots , k-1.
	$$ 
	By Cauchy-Schwarz we obtain
	for $r =1, \dots, k_1$, 
	$$\left \Vert \hat {f}^{k_1,\epsilon}_{t_{l_1}} \otimes_r \hat {f}^{k_2,\epsilon}_{t_{l_2}}\right \Vert_{\H^{k_1+k_2-2r}}
	\leq \left\Vert  \hat {f}^{k_1,\epsilon}_{t_{l_1}} \otimes_r  \hat {f}^{k_1,\epsilon}_{t_{l_1}} \;
	\right \Vert_{\H^{p-r}} \; \left \Vert  \hat {f}^{k_2,\epsilon}_{t_{l_2}} \otimes_r  \hat {f}^{k_2,\epsilon}_{t_{l_2}} \right\Vert_{\H^{q-r}}
	\to 0,$$
	for all $ t_{l_1},t_{l_2} \in [0,T], \,  1 \leq k_1 \leq n_0 < k_2 \leq M$.
	We can now apply an asymptotic independence result, Proposition \ref{proposition-spit-independence} in the Appendix, to conclude  the joint convergence
	in finite dimensional distributions of $X^\epsilon$ to $(X^W,X^Z)$. Furthermore, $X^W$ is independent of $X^Z$.

	The correlations between $X^i_t$ and $X^j_{t'}$,  where $i,j>n$, are $0$ if $m_i \not = m_j$, otherwise given by the $L^2$ norm of their integrands, which follows from the Wiener-It\^o isometry and are given by
	$$
	c_{i,m_i} c_{j,m_j} \int_{0}^t \! \! \!\int_{0}^{t'}  \! \! \int_{\R^{m_i}} \prod_{i=1}^{m_i} \left( s - \xi_i\right)_{+}^{\hat H(m) - \f 3 2} \prod_{i=1}^{m_i} \left( r - \xi_i\right)_{+}^{\hat H(m) - \f 3 2} d\xi_1 \dots \xi_{m_i} dr ds.
	$$

\textbf{Step 4, Convergence in the H\"older norms.}\\
We first choose $\gamma'\in \Big( \gamma,  (H^*(m_k) \wedge \f 1 2) - \f 1 {p_k}\Big)$. Then,
using the  Markov-Chebyshev inequality, we obtain that, as $M$ tends to $ \infty$,
\begin{align*}
\P\left( \vert X^{k,\epsilon} \vert_{\C^{\gamma'}([0,T],\R)} > M \right) &\leq \f {\Vert \theta|_{\C^{\gamma'}([0,T],\R)} \Vert_{p_k} } {M^{p_k}}\\
&\to 0,
\end{align*}
since $ \||\theta|_{ \C^{\gamma'}}\|_{p_k}  < \infty$ by Lemma \ref{Lp-bounds} and an application of  Kolmogorov's theorem \ref{Kolmogorov-theorem}.
 Furthermore, since $\C^{\gamma'([0,T];\R)} $ is compactly embedded in $\C^\gamma([0,T];\R)$ for $1>\gamma' > \gamma$,  by an type Arzela-Ascoli argument,  the sets $ \{ X^{k,\epsilon}_t : \vert X^{k,\epsilon} \vert_{\C^{\gamma'}([0,T],\R)} \leq M \} $ are sequentially compact in $\C^{\gamma}([0,T],\R)$. Hence, $\{X^{k,\epsilon}_t$\} is tight in $\C^{\gamma}([0,T],\R)$. See e.g. \cite{Friz-Victoir}. As tightness in each component implies joint tightness, we obtain that $X^{\epsilon}$ is tight in $\C^{\gamma}([0,T],\R^N)$, where $\gamma \in (0, \f 1 2 - \f {1} {\min_{k \leq n} p_k}) $ in case $0<n$ and $\gamma \in ( 0, \min_{k>n} H^*(m_k) - \f {1} {p_k} ) $ otherwise. 
By the above discussion, any limit of a converging subsequence has the same finite dimensional distributions, hence, every subsequence converges to the same limit. This concludes the proof  for the convergence in the H\"older norm.
\end{proof}

\section{Appendix: Joint convergence by asymptotic independence}
For the proof in the previous section we need the following which modifies results from \cite{Nourdin-Rosinski} and \cite{Nourdin-Nualart-Peccati}. Let $I_p(f)$ denote the $p^{th}$  iterated It\^o-Wiener integral of a symmetric function $f$ of $p$ variables,
$$I_p(f)= p!\int_{-\infty}^{\infty} \int_{-\infty}^{s_{p-1}}\dots \int_{-\infty}^{s_2} f(s_1, \dots, s_p)dW_{s_1} dW_{s_2} \dots dW_{s_p}.$$
If $f\in L^2(\R^p)$ and   $g\in L^2(\R^q)$ are symmetric functions and  $p, q\ge 1$,
 their $r^{th}$-contraction is given by
\begin{equation}\label{def-contraction}
f\otimes_r g=\int_{\R^r} f(x_1, \dots, x_{p-r}, s_1, \dots, s_r) g(y_1, \dots, y_{q-r}, s_1, \dots, s_r) ds_1 \dots ds_r,
\end{equation}
where $r \leq p \wedge q$.
If $f\otimes_1g=\int_{\R} f(x_1, \dots, x_{p-1}, s) g(y_1, \dots, y_{q-1}, s) ds$ vanishes,  so do all higher order contractions.

\begin{proposition}\label{proposition-spit-independence}
Let $q_1 \leq q_2 \leq \dots \leq q_n \leq p_1 \leq p_2 \leq \dots \le p_m$. Let  $f_i^\epsilon \in L^2(\R^{p_i})$, $
g_i^\epsilon \in L^2(\R^{q_i})$, $F^{\epsilon}=\left(I_{p_1}(f^{\epsilon}_1), \dots , I_{p_m}(f^{\epsilon}_m)\right)$ and  $G^{\epsilon}=\left(I_{q_1}(g^{\epsilon}_1), \dots, I_{q_n}(g^{\epsilon}_n)\right)$. Suppose that 
 for every $i,j$, and any $1 \leq r \leq q_i$:
$$ \Vert f^{\epsilon}_j \otimes_r g^{\epsilon}_i \Vert \to 0.$$
Then $F^\epsilon \to U$ and $G^{\epsilon} \to V$ weakly imply that $(F^{\epsilon},G^{\epsilon}) \to (U,V)$ jointly, where  $U$ and $V$ are taken to be independent random variables.
\end{proposition}

These results benefit from the insights of \"{U}st\"{u}nel-Zakai \cite{Ustunel-Zakai} on the independence of two iterated integrals  $I_p(f)$ and $I_q(g)$. They  are  independent   if and only if the 1-contraction between $f$ and $g$ vanishes almost surely with respect to the Lebesgue measure. 
An asymptotic independence result follows as below,  
  \begin{lemma}\cite[Thm. 3.1]{Nourdin-Nualart-Peccati} \label{contraction-covariance}
Let $F^{\epsilon}=I_p(f^{\epsilon})$ and $G^{\epsilon}=I_q(g^{\epsilon})$, where $f^\epsilon\in L^2(\R^p)$ and  $g^\epsilon\in L^2(\R^q)$ . Then,
$$\cov \left(  \left({F^{\epsilon}}\right)^2,\left({G^{\epsilon}}\right)^2 \right) \to 0$$
is equivalent to 
$ \Vert f^{\epsilon} \otimes_r g^{\epsilon}\Vert \to 0$,
for $1\leq r \leq p \wedge q$.
\end{lemma}

It is also known that if two  integrals  $I_p(f)$ and $I_q(g)$ are  independent, then 
their Malliavin derivatives are orthogonal, see \cite{Ustunel-Zakai}. This explains why Malliavin calculus  comes into prominent play,
which has been developed to its perfection in \cite[Lemma 3.2]{Nourdin-Nualart-Peccati}. 
Given a smooth test function $\phi$ we define,
$$ \Vert \phi \Vert_{q} = \Vert \phi \Vert_{\infty} + \sum_{\vert k \vert =1}^{q} \left\Vert \f { \partial^{k}} {\partial^k x} \right\Vert_{\infty},$$
where the sum runs over multi-indices $k=(k_1, \dots, k_m)$.
Let $L=-\delta D$ and 
throughout this section $f_i:\R^{p_i}\to \R$ and  $g:\R^q\to \R$  denote symmetric functions.
\begin{lemma}\cite{Nourdin-Nualart-Peccati}
\label{a.i.-key-ineq}
Let $q\leq p_i$, $g\in L^2(\R^{q})$, $G=I_q(g)$, $f_i\in L^2(\R^{p_i})$,  and $F_i= I_{p_i}(f_i)$ with $\E (F_i^2) = 1$. Set $F=(F_1, \dots, F_m)$ and let $\theta $ be a smooth test function. Then, 
$$ \E \,\left \vert \left\< (I-L)^{-1} \theta(F)DF_j,DG\right \>_{\H} \right\vert \leq \, c\,\Vert \theta \Vert_{q}\, \cov(F_j^2,G^2), $$
where $c$ is a constant depending on
 $\Vert F \Vert_{L^2}$, $\Vert G \Vert_{L^2} $,  and  $q$, $m$, $p_1, \dots, p_m$.
\end{lemma}

The final piece of the puzzle is the observation that
 the defect in being independent is quantitatively controlled by the covariance of the squares of the relative components. The following is from
 \cite{Nourdin-Nualart-Peccati}, our only modification is to take $G$ to be vector valued.
Let $g_i:\R^{q_i}\to \R$  be symmetric functions.

\begin{lemma}\label{nnp-extension}
Given $F=\left(I_{p_1}(f_1), \dots I_{p_m}(f_m)\right)$ and $G=\left(I_{q_1}(g_1), \dots, I_{q_n}(g_n)\right)$ such that $p_k \geq q_l$ for every pair of $k,l$.
Then, for all test functions $\phi$ and $ \psi$, the following holds for  some constant $c$, depending on $\Vert F \Vert_{L^2}$, $\Vert G \Vert_{L^2} $, and  $m$, $n$,  $p_1, \dots, p_m$, $q_1, \dots, q_n$,
$$  \E \left( \phi(F) \psi(G) \right) -\E \left( \phi(F) \right) \E \left( \psi(G) \right) \leq c  \Vert D\psi \Vert_{\infty} \Vert \phi \Vert_{q_n} \sum_{i=1}^{m} \sum_{j=1}^n \cov(F_i^2,G_j^2)$$
\end{lemma}

\begin{proof}   Define  $L^{-1} (\sum_{k=0}^\infty I_k(h_m))=\sum_{k=1}^\infty \f 1 k I_k(h_m) \in {\mathbb D}^{2,2}$.
The key equality is  $-DL^{-1}=(I-L)^{-1}D$. As in \cite{Nourdin-Nualart-Peccati},
\begin{align*}
\phi(F) - \E(\phi(F)) &= LL^{-1} \phi(F)= \sum_{j=1}^m \delta((I-L)^{-1} \partial_j \phi(F) DF_j).
\end{align*}
Multiplying both sides by $\psi(G)$, taking expectations and  using integration by parts we obtain
\begin{align*}
&\E \left( \phi(F) \psi(G) \right) -\E \left( \phi(F) \right) \E \left( \psi(G) \right) 
= \sum_{j=1}^m \sum_{i=1}^n \E \left( \< (I-L)^{-1} \partial_j \phi(F) DF_j,  D  G_i \>_{\H} \partial_i\psi(G) \right)\\
&\leq \Vert D\psi \Vert_{\infty} \sum_{j=1}^m \sum_{i=1}^n \left| \E \left( \< (I-L)^{-1} \partial_j \phi(F) DF_j, D G_i \>_{\H}  \right)\right|.
\end{align*}
To conclude, apply to each summand  Lemma \ref{a.i.-key-ineq} with $\theta = \partial_j \phi$ and $G=G_i$.
\end{proof}

\begin{lemma}\label{expectation-split}
Let $F^{\epsilon}=\left(I_{p_1}(f^{\epsilon}_1), \dots I_{p_m}(f^{\epsilon}_m)\right)$ and $G^{\epsilon}=\left(I_{q_1}(g^{\epsilon}_1), \dots, I_{q_n}(g^{\epsilon}_n)\right)$ with   $q_1 \leq q_2, \dots , q_n \leq p_1 \leq p_2\leq \dots\leq  p_m$. Then for every $i \leq m ,j \leq n$,
$$ \Vert f^{\epsilon}_j \otimes_r g^{\epsilon}_i \Vert \to 0, \quad 1\le r \le p_j\wedge q_i$$
implies that  for any smooth test functions  $\phi$ and $\psi$,
$$ \E \left( \psi(F^{\epsilon}) \psi(G^{\epsilon}) \right) - \E \left( \psi(F^{\epsilon}) \right) \E \left( \psi(G^{\epsilon})\right) \to 0.$$
\end{lemma}
\begin{proof}
Just combine Lemma \ref{nnp-extension} and Lemma \ref{contraction-covariance}.
\end{proof}
Finally, we  finish the proof of Proposition \ref{proposition-spit-independence}.
\begin{proof} 
Since $(F^{\epsilon},G^{\epsilon})$ is bounded in $L^2(\Omega)$ it is tight. Now choose a weakly converging subsequence $(F^n,G^n)$ with limit denoted by  $(X,Y)$. 
Let $\phi$ and $\psi$ be smooth test functions,  then
by  Lemma \ref{expectation-split}  and the bounds on $\phi, \psi$, we pass to the limit under the expectation sign and obtain 
$$\E\left( \phi(X) \psi(Y) \right)=\E\left( \phi(X)  \right) \E\left( \psi(Y) \right).$$
Thus every  limit measure is the product measure determined by $U$ and $V$, hence, $(F^{\epsilon},G^{\epsilon})$ converges as claimed.
\end{proof}

{\it Acknowledgement.} This research is partially funded by an EPSRC Roth studentship.


\begin{thebibliography}{MVN68}
	
	\bibitem[ATH12]{Al-Talibi-Hilbert}
	H.~Al-Talibi and A.~Hilbert.
	\newblock Differentiable approximation by solutions of {N}ewton equations
	driven by fractional {B}rownian motion.
	\newblock Preprint, 2012.
	
	\bibitem[BC09]{Buchmann-Ngai-bordercase}
	Boris Buchmann and Ngai~Hang Chan.
	\newblock Integrated functionals of normal and fractional processes.
	\newblock {\em Ann. Appl. Probab.}, 19(1):49--70, 2009.
	
	\bibitem[BGS19]{bourguin2019typical}
	Solesne Bourguin, Siragan Gailus, and Konstantinos Spiliopoulos.
	\newblock Typical dynamics and fluctuation analysis of slow-fast systems driven
	by fractional brownian motion, 2019.
	
	\bibitem[BH02]{BenHariz}
	Samir Ben~Hariz.
	\newblock Limit theorems for the non-linear functional of stationary {G}aussian
	processes.
	\newblock {\em J. Multivariate Anal.}, 80(2):191--216, 2002.
	
	\bibitem[Bil99]{Billingsley}
	Patrick Billingsley.
	\newblock {\em Convergence of probability measures}.
	\newblock Wiley Series in Probability and Statistics: Probability and
	Statistics. John Wiley \& Sons, Inc., New York, second edition, 1999.
	\newblock A Wiley-Interscience Publication.
	
	\bibitem[BM83]{Breuer-Major}
	Peter Breuer and P\'{e}ter Major.
	\newblock Central limit theorems for nonlinear functionals of {G}aussian
	fields.
	\newblock {\em J. Multivariate Anal.}, 13(3):425--441, 1983.
	
	\bibitem[BT05]{Boufoussi-Ciprian}
	Brahim Boufoussi and Ciprian~A. Tudor.
	\newblock Kramers-{S}moluchowski approximation for stochastic evolution
	equations with {FBM}.
	\newblock {\em Rev. Roumaine Math. Pures Appl.}, 50(2):125--136, 2005.
	
	\bibitem[BT13]{Bai-Taqqu}
	Shuyang Bai and Murad~S. Taqqu.
	\newblock Multivariate limit theorems in the context of long-range dependence.
	\newblock {\em J. Time Series Anal.}, 34(6):717--743, 2013.
	
	\bibitem[CFS82]{Cornfeld-Fomin-Sinai}
	I.~P. Cornfeld, S.~V. Fomin, and Ya.~G. Sina{i}.
	\newblock {\em Ergodic theory}, volume 245 of {\em Grundlehren der
		Mathematischen Wissenschaften [Fundamental Principles of Mathematical
		Sciences]}.
	\newblock Springer-Verlag, New York, 1982.
	\newblock Translated from the Russian by A. B. Sosinskii.
	
	\bibitem[CKM03]{Cheridito-Kawaguchi-Maejima}
	Patrick Cheridito, Hideyuki Kawaguchi, and Makoto Maejima.
	\newblock Fractional {O}rnstein-{U}hlenbeck processes.
	\newblock {\em Electron. J. Probab.}, 8:no. 3, 14, 2003.
	
	\bibitem[CNN20]{Campese-Nourdin-Nualart}
	Simon Campese, Ivan Nourdin, and David Nualart.
	\newblock Continuous breuer–major theorem: Tightness and nonstationarity.
	\newblock {\em Ann. Probab.}, 48(1):147--177, 01 2020.
	
	\bibitem[Dob79]{Dobrushin}
	R.~L. Dobrushin.
	\newblock Gaussian and their subordinated self-similar random generalized
	fields.
	\newblock {\em Ann. Probab.}, 7(1):1--28, 1979.
	
	\bibitem[EM02]{Embrechts-Maejima}
	Paul Embrechts and Makoto Maejima.
	\newblock {\em Selfsimilar processes}.
	\newblock Princeton Series in Applied Mathematics. Princeton University Press,
	Princeton, NJ, 2002.
	
	\bibitem[FGL15]{Friz-Gassiat-Lyons}
	Peter Friz, Paul Gassiat, and Terry Lyons.
	\newblock Physical {B}rownian motion in a magnetic field as a rough path.
	\newblock {\em Trans. Amer. Math. Soc.}, 367(11):7939--7955, 2015.
	
	\bibitem[FH14]{Friz-Hairer}
	Peter~K. Friz and Martin Hairer.
	\newblock {\em A course on rough paths}.
	\newblock Universitext. Springer, Cham, 2014.
	\newblock With an introduction to regularity structures.
	
	\bibitem[FK00]{Fannjiang-Komorowski-2000}
	Albert Fannjiang and Tomasz Komorowski.
	\newblock Fractional {B}rownian motions in a limit of turbulent transport.
	\newblock {\em Ann. Appl. Probab.}, 10(4):1100--1120, 2000.
	
	\bibitem[FV10]{Friz-Victoir}
	Peter~K. Friz and Nicolas~B. Victoir.
	\newblock {\em Multidimensional stochastic processes as rough paths}, volume
	120 of {\em Cambridge Studies in Advanced Mathematics}.
	\newblock Cambridge University Press, Cambridge, 2010.
	\newblock Theory and applications.
	
	\bibitem[GL19]{Gehringer-Li-homo}
	J.~Gehringer and Xue-Mei Li.
	\newblock Homogenization with fractional random fields.
	\newblock arXiv:1911.12600. This is now improved and split into `Functional
	limit theorem for fractional OU' and `Diffusive and rough homogenisation in
	fractional noise field', 2019.
	
	\bibitem[GL20]{Gehringer-Li-tagged}
	J.~Gehringer and Xue-Mei Li.
	\newblock Diffusive and rough homogenisation in fractional noise field.
	\newblock This is an improved version of arXiv:1911.12600 (part II), 2020.
	
	\bibitem[HBS]{Hurst}
	Harold~Edwin Hurst, R.~P. Black, and Y.M. Sinaika.
	\newblock {\em Long Term Storage in Reservoirs, An Experimental Study,}.
	\newblock Constable.
	
	\bibitem[HL20]{Hairer-Li}
	Martin Hairer and Xue-Mei Li.
	\newblock Averaging dynamics driven by fractional {B}rownian motion.
	\newblock {\em Ann. Probab.}, 48(4):1826--1860, 2020.
	
	\bibitem[JL77]{Jona-Lasinio}
	G.~Jona-Lasinio.
	\newblock Probabilistic approach to critical behavior.
	\newblock In {\em New developments in quantum field theory and statistical
		mechanics ({P}roc. {C}arg\`ese {S}ummer {I}nst., {C}arg\`ese, 1976)}, pages
	419--446. NATO Adv. Study Inst. Ser., Ser. B: Physics, 26, 1977.
	
	\bibitem[KLO12]{Komorowski-Landim-Olla}
	Tomasz Komorowski, Claudio Landim, and Stefano Olla.
	\newblock {\em Fluctuations in {M}arkov processes}, volume 345 of {\em
		Grundlehren der Mathematischen Wissenschaften [Fundamental Principles of
		Mathematical Sciences]}.
	\newblock Springer, Heidelberg, 2012.
	\newblock Time symmetry and martingale approximation.
	
	\bibitem[LCL07]{Lyons-Caruana-Levy}
	Terry~J. Lyons, Michael Caruana, and Thierry L\'{e}vy.
	\newblock {\em Differential equations driven by rough paths}, volume 1908 of
	{\em Lecture Notes in Mathematics}.
	\newblock Springer, Berlin, 2007.
	\newblock Lectures from the 34th Summer School on Probability Theory held in
	Saint-Flour, July 6--24, 2004, With an introduction concerning the Summer
	School by Jean Picard.
	
	\bibitem[Lyo94]{Lyons94}
	Terry Lyons.
	\newblock Differential equations driven by rough signals. {I}. {A}n extension
	of an inequality of {L}. {C}. {Y}oung.
	\newblock {\em Math. Res. Lett.}, 1(4):451--464, 1994.
	
	\bibitem[Mis08]{Mishura}
	Yuliya~S. Mishura.
	\newblock {\em Stochastic calculus for fractional {B}rownian motion and related
		processes}, volume 1929 of {\em Lecture Notes in Mathematics}.
	\newblock Springer-Verlag, Berlin, 2008.
	
	\bibitem[MT07]{Maejima-Ciprian}
	Makoto Maejima and Ciprian~A. Tudor.
	\newblock Wiener integrals with respect to the {H}ermite process and a
	non-central limit theorem.
	\newblock {\em Stoch. Anal. Appl.}, 25(5):1043--1056, 2007.
	
	\bibitem[MVN68]{Mandelbrot-VanNess}
	Benoit~B. Mandelbrot and John~W. Van~Ness.
	\newblock Fractional {B}rownian motions, fractional noises and applications.
	\newblock {\em SIAM Rev.}, 10:422--437, 1968.
	
	\bibitem[NNP16]{Nourdin-Nualart-Peccati}
	Ivan Nourdin, David Nualart, and Giovanni Peccati.
	\newblock Strong asymptotic independence on {W}iener chaos.
	\newblock {\em Proc. Amer. Math. Soc.}, 144(2):875--886, 2016.
	
	\bibitem[NNZ16]{Nourdin-Nualart-Zintout}
	Ivan Nourdin, David Nualart, and Rola Zintout.
	\newblock Multivariate central limit theorems for averages of fractional
	{V}olterra processes and applications to parameter estimation.
	\newblock {\em Stat. Inference Stoch. Process.}, 19(2):219--234, 2016.
	
	\bibitem[NP05]{Nualart-Peccati}
	David Nualart and Giovanni Peccati.
	\newblock Central limit theorems for sequences of multiple stochastic
	integrals.
	\newblock {\em The Annals of Probability}, 33(1):177--193, 2005.
	
	\bibitem[NP12]{Nourdin-Peccati}
	Ivan Nourdin and Giovanni Peccati.
	\newblock {\em Normal approximations with {M}alliavin calculus}, volume 192 of
	{\em Cambridge Tracts in Mathematics}.
	\newblock Cambridge University Press, Cambridge, 2012.
	\newblock From Stein's method to universality.
	
	\bibitem[NR14]{Nourdin-Rosinski}
	Ivan Nourdin and J.~Rosinski.
	\newblock Asymptotic independence of multiple wiener-it\^o integrals and the
	resulting limit laws.
	\newblock {\em The Annals of Probability}, 42(2):497--526, 2014.
	
	\bibitem[PT00]{Pipiras-Taqqu}
	Vladas Pipiras and Murad~S. Taqqu.
	\newblock Integration questions related to fractional {B}rownian motion.
	\newblock {\em Probab. Theory Related Fields}, 118(2):251--291, 2000.
	
	\bibitem[PT17]{Pipiras-Taqqu-book}
	Vladas Pipiras and Murad~S. Taqqu.
	\newblock {\em Long-range dependence and self-similarity}.
	\newblock Cambridge Series in Statistical and Probabilistic Mathematics, [45].
	Cambridge University Press, Cambridge, 2017.
	
	\bibitem[Ros61]{Rosenblatt}
	M.~Rosenblatt.
	\newblock Independence and dependence.
	\newblock In {\em Proc. 4th {B}erkeley {S}ympos. {M}ath. {S}tatist. and
		{P}rob., {V}ol. {II}}, pages 431--443. Univ. California Press, Berkeley,
	Calif., 1961.
	
	\bibitem[Sam06]{Samorodnitsky}
	Gennady Samorodnitsky.
	\newblock Long range dependence.
	\newblock {\em Found. Trends Stoch. Syst.}, 1(3):163--257, 2006.
	
	\bibitem[Sin76]{Sinai}
	Ja.~G. Sina{\i}.
	\newblock Self-similar probability distributions.
	\newblock {\em Teor. Verojatnost. i Primenen.}, 21(1):63--80, 1976.
	
	\bibitem[SV06]{Stroock-Varadhan-mult-dim-diffusion-processes}
	Daniel~W. Stroock and S.~R.~Srinivasa Varadhan.
	\newblock {\em Multidimensional diffusion processes}.
	\newblock Classics in Mathematics. Springer-Verlag, Berlin, 2006.
	\newblock Reprint of the 1997 edition.
	
	\bibitem[Taq79]{Taqqu}
	Murad~S. Taqqu.
	\newblock Convergence of integrated processes of arbitrary {H}ermite rank.
	\newblock {\em Z. Wahrsch. Verw. Gebiete}, 50(1):53--83, 1979.
	
	\bibitem[UZ89]{Ustunel-Zakai}
	Ali~S\"{u}leyman \"{U}st\"{u}nel and Moshe Zakai.
	\newblock On independence and conditioning on {W}iener space.
	\newblock {\em Ann. Probab.}, 17(4):1441--1453, 1989.
	
	\bibitem[You36]{Young36}
	L.~C. Young.
	\newblock An inequality of the {H}\"{o}lder type, connected with {S}tieltjes
	integration.
	\newblock {\em Acta Math.}, 67(1):251--282, 1936.
	
	\bibitem[Zha08]{Zhang-08}
	Songfu Zhang.
	\newblock Smoluchowski-{K}ramers approximations for stochastic equations with
	{L}\'evy noise.
	\newblock PhD thesis, Purdue University, 2008.
	
\end{thebibliography}
\end{document}